\documentclass[11pt]{amsart}
\usepackage{pgf}
\usepackage{pgfplots}
\usepackage{epstopdf}
\usepackage{amsthm,amsmath,amssymb,amsfonts,amscd,epsfig,mathrsfs}
\usepackage{graphicx,enumerate,csquotes,epigraph}
\usepackage[all]{xy}
\usepackage{booktabs,delarray,mathtools,float,faktor,xfrac}
\usepackage{textgreek,verbatim,subcaption,tabularx,longtable,multirow,spalign}
\usepackage{footnote}\makesavenoteenv{tabular}\makesavenoteenv{table}
\usepackage{hyperref}
\usepackage{tikz}
\usetikzlibrary{cd,positioning,arrows,shapes,calc,arrows.meta}

\pgfplotsset{compat=1.17}
\newtheorem{proposition}{Proposition}
\newtheorem{theorem}[proposition]{Theorem}
\newtheorem{definition}[proposition]{Definition}
\newtheorem{lemma}[proposition]{Lemma}

\newtheorem{corollary}[proposition]{Corollary}

\newtheorem{example}[proposition]{Example}

\newtheorem{remark}[proposition]{Remark}

\theoremstyle{plain}

\allowdisplaybreaks

\hypersetup{
	colorlinks=true,
	linkcolor=blue,
	filecolor=blue,
	urlcolor=blue,
	citecolor = blue,
}

\newcommand\DrawGenus[7]{
	\pgfmathsetmacro{\xstart}{#1 - (0.985*#4)}
	\pgfmathsetmacro{\ystart}{#2 + (0.2*#3)}
	
	\draw[color = #6, rotate around={#5:(#1,#2)}, #7] (\xstart, \ystart) arc (190:350:#4  and #3);
	\draw[color = #6, rotate around={#5:(#1,#2)}, #7] (\xstart, \ystart) arc (190:210:#4  and #3) arc (150:30:#4  and #3) arc (330:350:#4  and #3);
}

\newcommand\DrawDonut[7]{
	\pgfmathsetmacro{\fctr}{.08}
	\pgfmathsetmacro{\newwidth}{0.5*#4}
	\pgfmathsetmacro{\newheight}{0.5*#3}
	\draw[color = #6, rotate around={#5:(#1,#2)}, #7] (#1, #2) ellipse (#4  and #3);
	\DrawGenus{#1}{#2}{\newheight}{\newwidth}{#5}{#6}{#7}
}

\begin{document}

	\begin{abstract}
		In this article, we introduce $b$-semitoric systems as a generalization of semitoric systems, specifically tailored for $b$-symplectic manifolds.
		The objective of this article is to furnish a collection of examples and investigate the distinctive characteristics of these systems.
		A $b$-semitoric system is a $4$-dimensional $b$-integrable system that satisfies certain conditions: one of its momentum map components is proper and generates an effective global $S^1$-action, and all singular points are non-degenerate and devoid of hyperbolic components.
		To illustrate this concept, we provide five examples of $b$-semitoric systems by modifying the coupled spin oscillator and the coupled angular momenta, and we also classify their singular points.
		Additionally, we describe the dynamics of these systems through the image of their respective momentum maps.
	\end{abstract}
	
	\author{Joaquim Brugués}
	\address{Joaquim Brugués, Campus Middelheim, Gebouw G, M.G.215 Middelheimlaan 1 2020 Antwerp, Belgium, Laboratory of Geometry and Dynamical Systems, Universitat Polit\`{e}cnica de Catalunya, Avinguda del Doctor Mara\~{n}on 44-50, 08028, Barcelona}
	\email{joaquim.brugues@upc.edu, jbruguesmora@uantwerpen.be}
	
	\author{Sonja Hohloch}
	\address{Sonja Hohloch, Campus Middelheim, Gebouw G, M.G.211 Middelheimlaan 1 2020 Antwerp, Belgium}
	\email{sonja.hohloch@uantwerpen.be}
	
	\author{Pau Mir}
	\address{Pau Mir, Laboratory of Geometry and Dynamical Systems, Universitat Polit\`{e}cnica de Catalunya, Avinguda del Doctor Mara\~{n}on 44-50, 08028, Barcelona}
	\email{pau.mir.garcia@upc.edu}
	
	\author{Eva Miranda}
	\address{Eva Miranda, Laboratory of Geometry and Dynamical Systems and Institut de Matem\`atiques de la UPC-BarcelonaTech (IMTech), Universitat Polit\`{e}cnica de Catalunya, Avinguda del Doctor Mara\~{n}on 44-50, 08028, Barcelona \\ CRM Centre de Recerca Matem\`{a}tica, Campus de Bellaterra
		Edifici C, 08193 Bellaterra, Barcelona}
	\email{eva.miranda@upc.edu}
	
	\thanks{J.\ Brugués was fully and S.\ Hohloch partially funded by the FWO-FNRS Excellence of Science project G0H4518N “Symplectic techniques in differential geometry” with UA Antigoon Project-ID 36584}
	\thanks{P.\ Mir was funded in part by the Doctoral INPhINIT - RETAINING grant ID 100010434 LCF/BQ/DR21/11880025 of “la Caixa” Foundation}
	\thanks{J.\ Brugués, P.\ Mir and E.\ Miranda were partially supported by the Spanish State Research Agency AEI via the grant PID2019-103849GB-I00 of MCIN/ AEI /10.13039/501100011033 and by the Generalitat de Catalunya (AGAUR) via the grant 2021 SGR 00603.}
	\thanks{E. Miranda was supported by the Catalan Institution for Research and Advanced Studies via a 2021 ICREA Academia Prize   and by the Spanish State Research Agency, through the Severo Ochoa and Mar\'{\i}a de Maeztu Program for Centers and Units of Excellence in R\&D (project CEX2020-001084-M)}
	
	\title{Constructions of $b$-semitoric systems}
	
	\maketitle
	
	\section{Introduction}
	
	Hamiltonian systems are used to model many fundamental physical phenomena and are found in various fields in mathematics such as differential geometry, the calculus of variations and celestial mechanics, as well as in other sciences like physics, biology and engineering.
	Particularly, questions of conservation laws and symplectic rigidity  are naturally linked to Hamiltonian systems.
	
	A Hamiltonian system is called \emph{integrable} if it has the maximally possible number of independent conserved quantities.
	The field of integrable systems has a long tradition at the intersection of several disciplines, including dynamical systems, ODEs, PDEs, symplectic geometry, Lie theory, algebraic geometry, classical mechanics, and mathematical physics.
	Integrable systems exhibit some degree of predictability, as their flow lines remain in the fibers of the momentum map and are generically periodic or quasiperiodic.
	However, their global behavior can be complex.
	
	In recent decades, several classification schemes of integrable systems have been constructed based on  a number of invariants that capture various aspects of a system with respect to different notions of equivalence. These classification procedures provide an overview of all possible systems within a certain class and enable the distinction between non-equivalent systems. Integrable systems have a natural semilocal toric action associated with classical action-angle coordinates (see Arnold~\cite{arnold} and Duistermaat~\cite{Duistermaat80}). However, this action does not always extend globally. Notable classifications of symplectic type include the classification of toric systems by Atiyah~\cite{Atiyah82}, Guillemin and Sternberg~\cite{GuilleminSternberg82} and Delzant~\cite{Delzant88}.
	
	A special class of four-dimensional completely integrable systems is the class of \emph{semitoric systems}, where one of the first integrals is proper and generates a global $\mathbb{S}^1$-action, singularities are assumed to be non-degenerate, and none of the singularities exhibits hyperbolic components (see Definition \ref{def:semitoricsystem}).
	These systems were originally defined and studied by V\~{u} Ng\d{o}c in \cite{VuNgoc2003,VuNgoc2007}.
	These systems can be seen as a generalization of toric systems in dimension four with the important difference that only one of the first integrals is required to be periodic.
	This allows for the existence of focus-focus fibres, which in turn obstruct the global existence of action-angle coordinates, cf.\ Duistermaat~\cite{Duistermaat80}.
	
	Semitoric systems can be considered a $4$-dimensional first generalization of toric systems in the sense that one component of the momentum map induces an $\mathbb{S}^1$-action and the other one an $\mathbb{R}$-action on the manifold. For general integrable systems, both components induce $\mathbb{R}$-actions and can exhibit extremely complex behavior. For toric systems, all components give rise to $\mathbb{S}^1$-actions, leading to numerous constraints on the system as a whole.
	
	From a topological point of view, semitoric systems can be described using the theory of singular Lagrangian fibrations, cf.\ Bolsinov and Fomenko \cite{BolsinovFomenko04} and Zung~\cite{Zung03}.
	From the symplectic point of view, they were classified in terms of five symplectic invariants by Pelayo and V\~{u} Ng\d{o}c~\cite{PelayoVuNgoc2009,PelayoVuNgoc2011} for systems with maximally one focus-focus point per fiber.
	This restriction was later overcome by Palmer, Pelayo and Tang~\cite{PalmerPelayoTang19}.
	
	Semitoric systems appear naturally in physics, for example in the Jaynes-Cummings model (see Babelon and Cantini and Douçot \cite{BabelonCantiniDoucot09}) and the coupled angular momenta (see Sadovskii and Zhilinskii \cite{SadovskiiZhilinskii99}).
	During the past decade, semitoric systems were vividly studied: Hohloch, Sabatini and Sepe~\cite{HohlochSabatiniSepe2015} explained the relation of Pelayo and Vu Ngoc~\cite{PelayoVuNgoc2009,PelayoVuNgoc2011} classification with Karshon's~\cite{Karshon1999} classification of Hamiltonian $\mathbb{S}^1$-spaces.
	Alonso, Dullin and Hohloch~\cite{AlonsoDullinHohloch19, AlonsoDullinHohloch20} computed the invariants of the semitoric systems given by the coupled spin-oscillator and the coupled angular momenta.
	Hohloch and Palmer~\cite{HohlochPalmer18} generalized the coupled angular momenta system to a family of systems with two focus-focus points and Alonso and Hohloch~\cite{AlonsoHohloch21} computed the so-called \emph{height invariant} for a subfamily of this system.
	Le Floch and Palmer~\cite{LeFlochPalmer2018} generalized this method and found more examples by perturbing toric systems on Hirzebruch surfaces.
	De Meulenare and Hohloch~\cite{DeMeulenaereHohloch21} eventually constructed a family of systems with four focus-focus points which collide at a certain moment and form two focus-focus fibers with two focus-focus points in each.
	A survey article by Alonso and Hohloch~\cite{AlonsoHohloch19} gives an overview of the state of the art concerning examples and computations of classification invariants reached in 2019.
	Since then, allowing for hyperbolic singularities let to the definition of so-called hypersemitoric systems in Hohloch and Palmer~\cite{HohlochPalmer2021} and explicit examples by Gullentops and Hohloch~\cite{GullentopsHohloch2022}, which were both made possible by the study of parabolic points by Efstathiou and Giaccobe~\cite{EfstathiouGiacobbe2012}, Bolsinov, Guglielmi and Kudryavtseva~\cite{BolsinovGuglielmiKudryavtseva2018}, and Kudryavtseva and Martynchuk~\cite{KudryavtsevaMartynchuk2021existence,KudryavtsevaMartynchuk2021c}.
	
	A natural question regarding symplectic manifolds is their possible generalization to manifolds with boundary (cf.\ Nest and Tsygan~\cite{NestTsygan1996}) or, more generally, to Poisson manifolds that are symplectic away from a hypersurface (cf.\ Guillemin, Miranda and Pires~\cite{GuilleminMirandaPires2014,GuilleminMirandaPires2011} and Gualtieri and Li~\cite{GualtieriLi2014}).
	These structures are present in the literature under the name of $b^m$ or $\log$-symplectic manifolds (see also Miranda and Planas~\cite{MirandaPlanas2018}, Guillemin, Miranda, Pires and Scott~\cite{GuilleminMirandaPiresScott2017}, Guillemin, Miranda and Weitsman~\cite{GuilleminMirandaWeitsman2018convexity,GuilleminMirandaWeitsman2018quantization,GuilleminMirandaWeitsman2019}, Marcut and Osorno-Torres~\cite{MarcutOsorno2014obstructions,MarcutOsorno2014deformation} and Cavalcanti~\cite{Cavalcanti2017} for further inquiries in the topology, geometry and dynamics of $b^m$- or $\log$-symplectic manifolds).
	
	One of the most relevant examples where such singularities arise is the regularization of certain problems in celestial mechanics such as the restricted 3-body problem (cf.\ Kiesenhofer and Miranda~\cite{KiesenhoferMiranda17}, Kiesenhofer, Miranda and Scott~\cite{KiesenhoferMirandaScott2016}, Delshams, Kiesenhofer and Miranda~\cite{DelshamsKiesenhoferMiranda2017}, and Braddell, Delshams, Miranda, Oms and Planas~\cite{BDMOP}).
	More recently, further applications have been developed in the context of Painlevé transcendents (cf.\ Matveeva~\cite{Matveeva2022} and Matveeva and Miranda~\cite{MatveevaMiranda2022}).
	
	In this article we consider \emph{$b$-integrable systems}, the analogue of completely integrable systems for $b$-symplectic manifolds $(M^{2n},Z,\omega)$, first introduced by Guillemin, Miranda and Pires in \cite{GuilleminMirandaPires2014}.
	A $b$-integrable system is given by a \emph{$b$-function} $F:M\to\mathbb{R}^n$ which has maximal rank almost everywhere and whose components are in involution with respect to the bracket induced by the $b$-symplectic structure $\omega$.
	Among the particular features of a $b$-integrable system, there is the fact that they  cannot exhibit fixed points located at the critical hypersurface $Z$, as we discuss in this article.
	Action-angle coordinates for $b$-symplectic manifolds were investigated by Kiesenhofer, Miranda and Scott in \cite{KiesenhoferMirandaScott2016} and also by Miranda and Planas~\cite{MirandaPlanas2023} for $b^m$-symplectic manifolds and by Cardona and Miranda~\cite{CardonaMiranda} for their \emph{folded symplectic manifolds} analog. Action-angle coordinates for general Poisson manifolds were investigated by Laurent-Gengoux, Miranda and Vanhaecke in \cite{LaurentMirandaVanhaecke} however the action-angle coordinates are constructed for regular points of the Poisson structure.
	Kiesenhofer and Miranda~\cite{KiesenhoferMiranda17} provided a model for $b$-focus-focus singularities in dimension $6$.
	However, a global classification for these systems is still pending.
	The purpose of this article is to provide a collection of examples and examine specific features of these systems.
	
	The analog of toric systems to the context of $b$-symplectic manifolds was described by Guillemin, Miranda, Pires and Scott in \cite{GuilleminMirandaPiresScott2015} and by Gualtieri, Li, Pelayo and Ratiu in \cite{GualtieriLiPelayoRatiu2017}.
	In that context, a classification analogous to that of toric manifolds is completely developed by using formations akin to Delzant polytopes. 
	
	A particular type of $b$-integrable systems are \emph{$b$-semitoric systems}. We examine in this article the class of $b$-semitoric systems as a simultaneous generalization of two well-known classes of integrable systems: semitoric systems and $4$-dimensional $b$-toric systems. $b$-Semitoric systems are defined as $4$-dimensional $b$-integrable systems whose singular points are non-degenerate and contain no hyperbolic components. Therefore, the fixed points of a $b$-semitoric system can be only of elliptic-elliptic or focus-focus type and, by the result on non-existence of fixed points at the critical hypersurface $Z$, they are necessarily located away from $Z$.
	
	$b$-Semitoric systems generalize semitoric systems in the sense that they also include systems defined on manifolds in which the symplectic structure can have a singularity along a certain hypersurface.
	In this class of singular symplectic manifolds the variety of integrable systems with semitoric features that can be constructed is wider.
	On the other hand, $b$-semitoric systems generalize $4$-dimensional $b$-toric systems because the induced action is not given by an action of $\mathbb{T}^2$ any longer but, instead, it is just required to have one $\mathbb{S}^1$ component. This allows studying $4$-dimensional $b$-integrable systems which have not only elliptic-elliptic singularities but also fixed points of focus-focus type.
	
	In view of this, there are two natural ways to construct, from pre-existing systems, $b$-semitoric systems that are not just semitoric systems or $b$-toric systems. One way is to take a semitoric system $(M,\omega,F=(f_1,f_2))$, select a singular hypersurface $Z\subset M$ and replace the symplectic form $\omega$ by a $b$-symplectic form which is singular on $Z$. Then, one can modify $f_1$, $f_2$ and associate to them $b$-functions. If this is done in an appropriate way, a $b$-semitoric system is produced from the semitoric system. The other way is to take a $4$-dimensional $b$-toric system $(M,Z,\omega,F=(f_1,f_2))$ and perturb either $f_1$ or $f_2$ in a way that the system is still $b$-integrable with singularities of focus-focus type.
	
	In this article, we take as starting point two pre-existing semitoric systems to construct five different $b$-semitoric systems. On the one hand, we take the coupled spin oscillator, a particular case of the Jaynes-Cummings~\cite{JaynesCummings63} model from quantum optics consisting of the coupling of a classical spin on the two-sphere $\mathbb{S}^2$ with a harmonic oscillator on the plane $\mathbb{R}^2$, see e.g.\ Pelayo and V\~{u} Ng\d{o}c~\cite{PelayoVuNgoc12}. We modify it in two different ways to create two $b$-semitoric systems, the \emph{$b$-coupled spin oscillator} and the \emph{reversed $b$-coupled spin oscillator}.
	
	On the other hand, we take the $1$-parameter family of the coupled angular momenta system, the classical version of the addition of two quantum angular momenta, defined on the product of two copies of $\mathbb{S}^2$.
	It models, for example, the reduced Hamiltonian of a hydrogen-like atom in the presence of parallel electric and magnetic fields, cf.\ Sadovskii, Zhilinskii and Michel~\cite{SadovskiiZhilinskiiMichel96}.
	We modify it in three different ways to create three families of $b$-semitoric systems.
	While the original $1$-parameter family of the coupled angular momenta is an interpolation between a toric system and a semitoric system of toric type, two of our systems interpolate between a $b$-toric system and a $b$-semitoric system.
	
	For all the constructed examples, after proving that they are honest $b$-integrable systems, we classify their fixed points and describe their moment map. By doing so, we completely characterize their dynamics. In the cases of the \emph{$b$-coupled spin oscillator} and the \emph{reversed $b$-coupled spin oscillator} the local analysis to determine the type of the singular points is self-contained in the article. In the cases of the $b$-coupled angular momenta, to determine the type of the singular points we combine the local analysis of these new systems with the already known classification of the singular points of the original coupled angular momenta system.
	
	Our analysis of the different examples of $b$-semitoric systems opens the door to finding a global classification of $b$-semitoric systems. It will presumably depend on a number of invariants similar to the symplectic invariants of the semitoric classification of Pelayo and V\~{u} Ng\d{o}c~\cite{PelayoVuNgoc2009, PelayoVuNgoc2011} and will also take into account the constraints imposed by the underlying $b$-symplectic structure. The classification scheme by Braddell, Kiesenhofer and Miranda \cite{BraddellKiesenhoferMiranda2023} will be key.
	
	Our article is also a good starting point to study other connected models from a $b$-symplectic angle, especially $b$-semitoric models coming from well-known semitoric models such as the octagon (cf.\ De Meulenaere and Hohloch~\cite{DeMeulenaereHohloch21}), the spherical pendulum the champagne bottle and, in general, $4$-dimensional $b$-integrable systems that carry and $\mathbb{S}^1$-action. It also connects with a potential generalization of hypersemitoric systems, introduced by Gullentops and Hohloch~\cite{GullentopsHohloch2022}, to $b$-hypersemitoric systems.
	
	\subsection*{Organization of the article} In Section \ref{Section:preliminaries} we recall the basic results on integrable systems and on the classification of their singularities.
	We provide the basic definitions of toric and semitoric systems and we describe the two examples that we use to construct $b$-semitoric systems, the coupled spin oscillator and the coupled angular momenta.
	We also give the definition of $b$-integrable and $b$-toric systems in this section.
	In Section \ref{sec:bsemitoric} we define the notion of $b$-semitoric system and we prove that it contains no fixed points in $Z$.
	In Section \ref{section:spin-oscillator} we modify the coupled spin oscillator system to construct two $b$-semitoric systems: the \emph{$b$-coupled spin oscillator} and the \emph{reversed $b$-coupled spin oscillator}.
	We identify their singular points, classify them and characterize the image of their momentum maps.
	In Section \ref{section:cam} we modify the $1$-parameter family of the coupled angular momenta system to construct three families of $b$-semitoric systems.
	We also identify and classify their fixed points and compute the image of their momentum maps.
	
	
	\section{Preliminaries}
	\label{Section:preliminaries}
	
	In this section we summarize the main results from the literature that we need throughout the article. We include the basics on completely integrable systems and the particular families of toric and semitoric systems. We also recall the essential notions of $b$-symplectic geometry and $b$-toric systems.
	
	\subsection{Singular points of integrable systems}
	
	Throughout this article, we work with $4$-dimensional symplectic manifolds and in this preliminaries section some of the classical results are adapted to the $4$-dimensional setting, while for others we prefer to keep their general version.
	
	\begin{definition}
		Let $(M^{2n},\omega)$ be a symplectic manifold and let $F=(f_1, \dots, f_n) : M \to \mathbb{R}^n$. The triplet $(M^{2n},\omega,F)$ is said to be a \emph{completely integrable system} if $F$ is a smooth map such that $dF$ has maximal rank almost everywhere and the components $f_i$ are in involution, i.e.\ $\{f_i, f_j\} = 0$ for all $i,j$.
		\label{def:integrablesystem}
	\end{definition}
	
	The map $F = (f_1, \dots , f_n)$ is called \emph{momentum map} and its flow is given by the concatenation of the flows of $f_1,\dots,f_n$ induces a group action of $\mathbb{R}^n$. A point $p \in M$ is \emph{regular} if $dF(p)$ has maximal rank and \emph{singular} (or \emph{critical}) if the rank of $dF(p)$ is lower than $n$. The set $F^{-1}(c)\subset M$ is referred to as the \emph{fiber} over $c\in\mathbb{R}^n$. The connected components of a fiber are called \emph{leaves} and the Arnold-Liouville-Mineur theorem fully describes the dynamics on the ones which are regular, while the understanding of singular fibres, i.e.\ those containing at least one singular point, is not complete and, hence, an active research field.
	
	We are interested in the non-degenerate singular points, for which normal forms were established in the works of Rüssmann~\cite{Russmann64}, Vey~\cite{Vey78}, Colin de Verdière and Vey~\cite{ColindeVerdiereVey79}, Eliasson~\cite{Eliasson84, Eliasson90}, Dufour and Molino~\cite{DufourMolino88}, Miranda~\cite{Miranda2003,Miranda2014}, Miranda and Zung~\cite{MirandaZung04}, Miranda and V\~{u} Ng\d{o}c~\cite{MirandaVuNgoc05}, V\~{u} Ng\d{o}c and Wacheux~\cite{VuNgocWacheux13} and Chaperon~\cite{Chaperon13}.
	
	\begin{theorem}[Local normal form for non-degenerate singularities]
		Consider a $2n$-dimensional completely integrable system $(M,\omega,F=(f_1,\dots,f_n))$ and let $p\in M$ be a non-degenerate singular point. Then
		\begin{enumerate}
			\item there exists an open neighbourhood $U\subset M$ of $p$, local symplectic coordinates $(x_1,\dots,x_n,\xi_1,\dots,\xi_n)$ on $U$ and smooth functions $q_1,\dots,q_n:U\to\mathbb{R}$ such that $p$ corresponds to the origin in these coordinates, $\{q_i,f_j\}=0$ for all $i,j \in \{1,\dots,n\}$ and each $q_i$ is of one of the following forms:
			\begin{itemize}
				\item $q_i = (x_i^2+\xi_i^2)/2$ {\em (elliptic component),}
				\item $q_i = x_i \xi_i$ {\em (hyperbolic component),}
				\item $q_i = x_i \xi_{i+1} - x_{i+1} \xi_i$ and $q_{i+1} = x_i \xi_i + x_{i+1} \xi_{i+1}$ {\em (focus-focus component),}
				\item $q_i = \xi_i$ {\em (regular component).}
			\end{itemize}
			\item If there are no hyperbolic components, then the system of equations $\{q_i,f_j\}=0$ for $ i,j \in \{1,\dots,n\}$ is equivalent to the existence of a local diffeomorphism $g : \mathbb{R}^n\to\mathbb{R}^n$ such that
			$$g \circ f = (q_1,\dots,q_n) \circ (x_1,\dots, x_n, \xi_1,\dots,\xi_n).$$
		\end{enumerate}
		\label{thm:localnf}
	\end{theorem}
	
	The number of elliptic, hyperbolic, and focus-focus components locally classifies a non-degenerate singular point and is referred to as its \emph{Williamson type}. In view of this, integrable systems on 4-dimensional manifolds admit exactly six possible types of non-degenerate singular points:
	
	\begin{itemize} 
		\item
		{\em rank 0 (fixed points):} 
		elliptic-elliptic,
		focus-focus,
		hyperbolic-hyperbolic,
		hyperbolic-elliptic.
		\item
		{\em rank 1 (one-dimensional orbits of the induced action):}
		elliptic-regular,
		hyperbolic-regular.
	\end{itemize}
	
	In Sections \ref{section:spin-oscillator} and \ref{section:cam} we classify the singular points of a number of integrable systems defined in $4$-dimensional manifolds. For the classification we follow the recipe of Bolsinov and Fomenko~\cite{BolsinovFomenko04} and use their notation conventions, which we introduce next.
	
	\begin{definition}
		\cite[Bolsinov and Fomenko, Definition 1.22]{BolsinovFomenko04}
		A fixed point of the completely integrable system $(M,\omega,F=(f_1,f_2))$ is non-degenerate if the Lie algebra $K(f_1,f_2)$ generated by the linear parts of the Hamiltonian vector fields $X_{f_1}$ and $X_{f_2}$ is a Cartan subalgebra in $\mathfrak{sp}(4,\mathbb{R})$.
		\label{def:nondegfixedpoints}
	\end{definition}
	
	The Lie algebra $K(f_1, f_2)$ can be described in terms of $f_1$ and $f_2$ and, in particular, in terms of their quadratic parts, i.e., the Hessians $d^2f_1$ and $d^2f_2$. They generate the linear symplectic operators $A_{f_1} = \Omega^{-1} d^2f_1$ and $A_{f_2}= \Omega^{-1} d^2f_2$, where $\Omega$ is the matrix of the symplectic structure, which coincide with the linearizations of $X_{f_1}$ and $X_{f_2}$ at the singular point.
	
	To check if the algebra $K(f_1,f_2)$ generated by the linear operator $c_1 A_{f_1} + c_2 A_{f_2}$ is a Cartan subalgebra in $\mathfrak{sp}(4,\mathbb{R})$, it has to be first checked that it is two-dimensional and contains an element whose eigenvalues are all different. Then, it has to be proved that it is conjugate to one of the next four Cartan subalgebras of $\mathfrak{sp}(4,\mathbb{R})$ (classified by Williamson~\cite{Williamson36}):
	\begin{equation}
		\label{eq:classification_cartan}
		\begin{pmatrix}
			0 & 0 & -\alpha & 0\\
			0 & 0 & 0 & -\beta\\
			\alpha & 0 & 0 & 0\\
			0 & \beta & 0 & 0
		\end{pmatrix}
		\begin{pmatrix}
			-\alpha & 0 & 0 & 0\\
			0 & 0 & 0 & -\beta\\
			0 & 0 & \alpha & 0\\
			0 & \beta & 0 & 0
		\end{pmatrix}
		\begin{pmatrix}
			-\alpha & 0 & 0 & 0\\
			0 & -\beta & 0 & 0\\
			0 & 0 & \alpha & 0\\
			0 & 0 & 0 & \beta
		\end{pmatrix}
		\begin{pmatrix}
			-\alpha & -\beta & 0 & 0\\
			\beta & -\alpha & 0 & 0\\
			0 & 0 & \alpha & -\beta\\
			0 & 0 & \beta & \alpha
		\end{pmatrix},
	\end{equation}
	where $\alpha,\beta\in\mathbb{R}$.
	
	Then, the type of a non-degenerate fixed point $p$ is determined by the conjugacy class of $c_1 A_{f_1} + c_2 A_{f_2}$ or, in practice, by its eigenvalues $\lambda_1,\lambda_2,\lambda_3,\lambda_4$ (see Williamson~\cite{Williamson36}, Eliasson~\cite{Eliasson90} and Miranda and Zung~\cite{MirandaZung04}). In particular, the type of $p$ is:
	\begin{itemize}
		\item 
		\emph{elliptic-elliptic:} four imaginary eigenvalues $\{\lambda_1, \lambda_2\} = \{\pm i \alpha\}$ and $\{\lambda_3, \lambda_4 \} = \{\pm i\beta\}$,
		\item 
		\emph{elliptic-hyperbolic:} two real and two imaginary eigenvalues$\{\lambda_1, \lambda_2\} = \{\pm i \alpha\}$ and $\{\lambda_3, \lambda_4\} = \{\pm \beta\}$,
		\item 
		\emph{hyperbolic-hyperbolic:} four real eigenvalues $\{\lambda_1, \lambda_2\} = \{\pm \alpha\}$ and $\{\lambda_3, \lambda_4\} = \{\pm \beta\}$,
		\item 
		\emph{focus-focus:} four complex eigenvalues $\{\lambda_1, \lambda_2, \lambda_3, \lambda_4\} =  \{\pm  \alpha \pm i\beta\}$
	\end{itemize}
	with $\alpha, \beta \in \mathbb{R}^{\neq 0}$ and $\alpha \neq \beta$ for the elliptic-elliptic and hyperbolic-hyperbolic cases.
	
	Summarizing, the non-degeneracy and the type of a fixed point can be determined through the eigenvalues of a linear operator $A$ given by a linear combination of $A_{f_1}=\Omega^{-1} d^2f_1$ and $A_{f_2}=\Omega^{-1} d^2f_2$ whose eigenvalues are all different. The matrix of $A$ will be conjugate to one of the above matrices and its spectrum will be one of the four above types.
	
	\subsection{Toric and semitoric systems}
	
	A group action is called \emph{effective} if the neutral element is the only one acting trivially.
	The class of completely integrable  toric systems is the most accessible and it is natural to try to extend their classification to more general systems.
	
	\begin{definition}
		A $2n$-dimensional completely integrable system $(M,\omega,\mathcal{F})$ is \emph{toric} if the flow of $\mathcal{F}$ generates an effective action of $\mathbb{T}^n$ on $M$.
		\label{def:toricsystem}
	\end{definition}
	
	Now consider the following class of polytopes:
	
	\begin{definition}
		A convex polytope $\Delta \subset \mathbb{R}^n$ is a \emph{Delzant polytope} if it is:
		\begin{enumerate}
			\item \emph{simple}, i.e.\ exactly $n$ edges meet at each vertex,
			\item \emph{rational}, i.e., all edges have rational slope, meaning, they are of the form $p+vt$  where $p\in \mathbb{R}^n$ is the vertex, $v \in \mathbb{Z}^n$ is the directional vector of the given edge, and $t \in \mathbb{R}$ and
			\item \emph{smooth}, i.e.\ at each vertex, the directional vectors of the meeting edges form a basis for $\mathbb{Z}^n$.
		\end{enumerate}
		\label{def:delzantpolytope}
	\end{definition}
	
	Toric systems on compact connected manifolds are completely classified in terms of the images of their momentum map:
	
	\begin{theorem}[Delzant~\cite{Delzant88}]
		Up to symplectic equivariance, any toric system $(M,\omega,F)$ on a compact connected symplectic $2n$-dimensional manifold $(M, \omega)$ is determined by $F(M)$, which is a Delzant polytope. Conversely, for any Delzant polytope $\Delta$, there exists a compact connected symplectic $2n$-dimensional manifold $(M, \omega)$ and a momentum map $F: M \to \mathbb{R}^n$ such that $(M,\omega,F)$ is toric with $F(M)=\Delta$.
		\label{thm:Delzant}
	\end{theorem}
	
	Toric systems do not admit singular points with hyperbolic or focus-focus components, just elliptic ones. In dimension $4$, a natural generalization of toric systems is the wider class of semitoric systems, which allows for the existence of more complicated fixed points.
	
	\begin{definition}
		A $4$-dimensional completely integrable system $(M,\omega, F=(L,H))$ is \emph{semitoric} if
		\begin{enumerate}
			\item $L$ is proper and generates an effective $\mathbb{S}^1$-action on $M$ and
			\item all singular points of $F=(L,H)$ are non-degenerate and do not include hyperbolic components.
		\end{enumerate}
		\label{def:semitoricsystem}
	\end{definition}
	
	Semitoric systems can have singular points of three types: elliptic-elliptic, focus-focus, and elliptic-regular. A semitoric system $(M,\omega,F=(L,H))$ is \emph{simple} if there is at most one focus-focus point in each fiber of $L$. Simple semitoric systems were classified by Pelayo and V\~{u} Ng\d{o}c~\cite{PelayoVuNgoc2009, PelayoVuNgoc2011} in terms of five invariants: the number of focus-focus points, the semitoric polygon, the height invariant, the Taylor series invariant, and the twisting index invariant. This classification was extended to non-simple systems by Palmer, Pelayo and Tang~\cite{PalmerPelayoTang19}. The first three invariants were already developed by V\~{u} Ng\d{o}c in \cite{VuNgoc2007} and in the work of Le Floch and Palmer~\cite{LeFlochPalmer2018} they are fit together into a single invariant called the \emph{marked semitoric polygon invariant}. The Taylor series invariant was also constructed by V\~{u} Ng\d{o}c~\cite{VuNgoc03}. In general, the behavior of semitoric systems is much more complicated due to the presence of focus-focus singularities. Different from toric systems where the classifying momentum polytope is determined by its finite number of vertices, semitoric systems may depend on infinitely many data, in particular when the Taylor series invariant has infinitely many non-vanishing coefficients.
	The two examples of semitoric systems that we turn into $b$-semitoric systems are the coupled spin-oscillator and the coupled angular momenta.
	
	\subsubsection{The coupled spin-oscillator}
	\label{sec:classicalcso}
	
	One of the simplest examples of semitoric systems is the coupling of a classical spin on the $2$-sphere $\mathbb{S}^2$ with a harmonic oscillator in the plane $\mathbb{R}^2$. The classical system is a simplification of the Jaynes-Cummings model (see Babelon, Cantini and Doucot~\cite{BabelonCantiniDoucot09}) and was studied in detail by Pelayo and V\~{u} Ng\d{o}c in \cite{PelayoVuNgoc12} and by Alonso, Dullin and Hohloch in \cite{AlonsoDullinHohloch19}.
	
	Let $\rho_1,\rho_2>0$ be positive constants. Consider the product manifold $M=\mathbb{S}^2 \times \mathbb{R}^2$ with Cartesian coordinates $(x,y,z)$ on the unit sphere $\mathbb{S}^2 \subset \mathbb{R}^3$ and Cartesian coordinates $(u,v)$ on the plane $\mathbb{R}^2$. Consider the symplectic form $\omega = -\rho_1\, \omega_{\mathbb{S}^2} + \rho_2\, \omega_{\mathbb{R}^2}$ on $M$, where $\omega_{\mathbb{S}^2}$ and $\omega_{\mathbb{R}^2}$ are the standard symplectic structures on $\mathbb{S}^2$ and $\mathbb{R}^2$ respectively.
	
	\begin{definition}
		The \emph{coupled spin-oscillator} is a $4$-dimensional Hamiltonian integrable system $(M,\omega,F=(L,H))$, where
		\begin{equation*}
			\begin{cases}
				L(x,y,z,u,v)=\rho_1 z + \rho_2 \frac{u^2+v^2}{2} ,\\
				H(x,y,z,u,v)=\frac{xu+yv}{2} .
			\end{cases}
		\end{equation*}
		\label{def:coupledspin-oscillator}
	\end{definition}
	
	The coupled spin-oscillator system is completely integrable and semitoric (see Pelayo and V\~{u} Ng\d{o}c~\cite{PelayoVuNgoc12}). The map $L$ is the momentum map for the simultaneous rotations of the sphere around its vertical axis and of the plane around the origin. The map $H$ measures the difference between the polar angles on the sphere and on the plane. It has one focus-focus singularity (at the point $m:=(0,0,1,0,0)$), one elliptic-elliptic singularity (at the point $p:=(0,0,-1,0,0)$) and two one-parameter families of elliptic-regular singularities emanating from $p$. The image of the moment map of the coupled spin-oscillator is depicted in Figure \ref{fig:momentmapcso}.
	
	\begin{figure}[ht!]
		\begin{tikzpicture}
			\begin{axis}[ticks=none,
				axis y line=center,
				axis x line=middle,
				axis on top,xmin=-1,xmax=5,
				domain=1:11,xlabel=$L$,ylabel=$H$]
				\addplot[fill,mark=none,blue!10,samples=100]
				({(x^2-3)/(2*x)},{(x^2-1)/(2*x^(3/2))}) \closedcycle;
				\addplot[mark=none,blue,samples=100]
				({(x^2-3)/(2*x)},{(x^2-1)/(2*x^(3/2))});
				\addplot[fill,mark=none,blue!10,samples=100]
				({(x^2-3)/(2*x)},{-(x^2-1)/(2*x^(3/2))}) \closedcycle;
				\addplot[mark=none,blue,samples=100]
				({(x^2-3)/(2*x)},{-(x^2-1)/(2*x^(3/2))});
				\addplot[blue, mark=*,mark options={mark size=2pt}] coordinates {(-1,0)};
				\addplot[red, mark=*,mark options={mark size=2pt}] coordinates {(1,0)};
			\end{axis}
		\end{tikzpicture}
		\caption{Image of the momentum map of the coupled spin-oscillator.  The blue dot is the image of the elliptic-elliptic singularity and the red dot is the image of the focus-focus singularity.}
		\label{fig:momentmapcso}
	\end{figure}
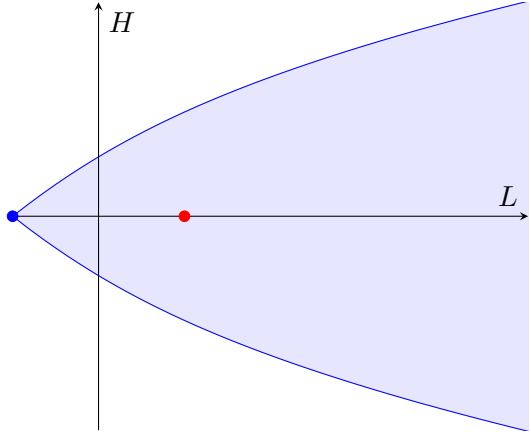
	
	On $\mathbb{S}^2$ we will use Cartesian coordinates $(x,y,z)$ with the assumption that $x^2+y^2+z^2=1$. Nevertheless, to make explicit computations it is more convenient to use Cartesian coordinates $(x,y)$ away from the equator and to use cylindrical coordinates $(\theta,z)$, defined by
	\begin{equation*}
		\begin{cases}
			z = \pm\sqrt{1-x^2-y^2} ,\\
			\theta = \arg\left(x+iy\right),
		\end{cases}
	\end{equation*}
	away from the poles.
	
	Explicitly, the appropriate charts to work away of the equator $z=0$ are $(\varphi,U^+)$ and $(\varphi,U^-)$, where $U^+=\{(x,y,z)\in\mathbb{S}^2\mid z>0\}\subset \mathbb{S}^2$, $U^-=\{(x,y,z)\in\mathbb{S}^2\mid z<0\}\subset \mathbb{S}^2$ and
	$$\begin{array}{rccc}
		\varphi:&\mathbb{S}^2\subset\mathbb{R}^3&\longrightarrow&\mathbb{R}^2\\
		&(x,y,z)&\longmapsto&(x,y) .
	\end{array}$$
	To work away from the poles $z=\pm 1$ an appropriate chart is $(\phi,U^0)$, where $U^0=\{(x,y,z)\in\mathbb{S}^2\mid \lvert z\rvert<1\}\subset \mathbb{S}^2$ and
	$$\begin{array}{rccc}
		\phi:&\mathbb{S}^2\subset\mathbb{R}^3&\longrightarrow&\mathbb{S}^1\times\mathbb{R}^1\\
		&(x,y,z)&\longmapsto&(\theta,z).
	\end{array}$$
	
	These coordinate charts together with the standard $(u,v)$ coordinate charts on $\mathbb{R}^2$ extend to natural charts on $\mathbb{S}^2 \times \mathbb{R}^2$. The symplectic form $\omega= -\rho_1\, \omega_{\mathbb{S}^2} + \rho_2\, \omega_{\mathbb{R}^2}=$ writes as
	$$\omega = -\rho_1\frac{1}{\pm\sqrt{1-x^2-y^2}}dx\wedge dy + \rho_2\,du \wedge dv$$
	on $M^{\pm}:=U^{\pm}\times\mathbb{R}^2$, and as
	$$\omega = -\rho_1\, d\theta \wedge dz + \rho_2\,du \wedge dv,$$
	on $M^0:=U^0\times\mathbb{R}^2$.
	
	The functions $L$ and $H$ can be rewritten on $M^{\pm}$ as
	\begin{equation*}
		\begin{cases}
			L(x,y,u,v)= \pm\rho_1\sqrt{1-x^2-y^2} + \rho_2 \frac{u^2+v^2}{2}\\
			H(x,y,u,v)= \frac{1}{2}\left(xu+yv\right)
		\end{cases},
	\end{equation*}
	and, on $M^0$, as
	\begin{equation*}
		\begin{cases}
			L(z,\theta,u,v)= \rho_1 z + \rho_2 \frac{u^2+v^2}{2}\\
			H(z,\theta,u,v)= \frac{\sqrt{1-z^2}}{2}\left(u\cos\theta + v\sin\theta\right)
		\end{cases}.
	\end{equation*}
	
	\subsubsection{The coupled angular momenta}
	\label{subsec:classicalcam}
	The coupling of two quantum angular momenta was studied by Sadovksii and Zhilinskii~\cite{SadovskiiZhilinskii99} and the classical version of the same system is a well-known compact semitoric system which has been studied in detail by Hohloch and Palmer~\cite{HohlochPalmer18}, Le Floch and Pelayo~\cite{LeFlochPelayo19} and Alonso, Dullin and Hohloch~\cite{AlonsoDullinHohloch20}, and in particular its invariants by Alonso and Hohloch~\cite{AlonsoHohloch21}.
	
	Consider $M=\mathbb{S}^2 \times \mathbb{S}^2$ and endow it with the symplectic form $\omega = -(R_1 \omega_{\mathbb{S}^2} + R_2 \omega_{\mathbb{S}^2})$, where $\omega_{\mathbb{S}^2}$ is the standard symplectic form of $\mathbb{S}^2$ and $0<R_1<R_2$ are constants. 
	
	Let $(x_i,y_i,z_i)$ be Cartesian coordinates on the unit sphere $x_i^2 + y_i^2 + z_i^2 = 1$, where $i \in \{1,2\}$ and consider a parameter $t \in \mathbb{R}$. The coupled angular momenta is  the family of $4$-dimensional completely integrable systems parameterized by $t$ and defined by
	\begin{equation}
		\begin{cases}
			L(x_1,y_1,z_1,x_2,y_2,z_2)\;:= R_1z_1 + R_2z_2,\\
			H(x_1,y_1,z_1,x_2,y_2,z_2):= (1-t) z_1 + t (x_1x_2 + y_1y_2 + z_1z_2).
			\label{eq:classicalcam}
		\end{cases}
	\end{equation} 
	
	The system has $4$ fixed points at $p_{\pm,\pm}=(0,0,\pm 1,0,0,\pm 1)$, see Sadovksii and Zhilinskii~\cite{SadovskiiZhilinskii99}. All of them are of elliptic-elliptic type for all values of $t$ except for $p_{+,-}=(0,0,1,0,0,-1)$, which is non-degenerate and of elliptic-elliptic type if $t<t^-$ or $t>t^+$, of focus-focus type for $t^- < t < t^+$ and degenerate for $t\in\{t^-,t^+\}$, where
	\begin{equation*}
		t^\pm = \frac{R_2}{2R_2+R_1\mp 2\sqrt{R_1 R_2}}.
	\end{equation*} 
	
	It can be shown that $0<t^- < \frac{1}{2}  < t^+\leq 1$, meaning that for the value $t=\frac{1}{2}$ there is always a focus-focus singularity.
	
	Throughout the section, it will be useful to work with different charts on $M=\mathbb{S}^2 \times \mathbb{S}^2$. We are interested in proving the global properties of the system, for which the double cylindrical chart is well suited, but also in studying local behaviours around the fixed points at the double poles, for which double Cartesian charts are better.
	
	To study the system away from the poles $z_1,z_2=\pm 1$, it is useful to rewrite it using the double cylindrical coordinate chart $(\phi_1,U_1^0)\times(\phi_2,U_2^0)$, where $U_i^0=\{(x_i,y_i,z_i)\in\mathbb{S}^2\mid \lvert z\rvert<1\}\subset \mathbb{S}^2$ for $i \in \{1,2\}$ and
	$$\begin{array}{rccc}
		\phi_i:&\mathbb{S}^2\subset\mathbb{R}^3&\longrightarrow&\mathbb{S}^1\times\mathbb{R}^1\\
		&(x_i,y_i,z_i)&\longmapsto&(\theta_i,z_i).
	\end{array}$$
	In these coordinates the symplectic form $\omega = -(R_1 \omega_{\mathbb{S}^2} + R_2 \omega_{\mathbb{S}^2})$ is
	$$\omega = -R_1\, d\theta_1 \wedge dz_1 - R_2\, d\theta_2 \wedge dz_2,$$
	and the system writes as:
	\[\left\{ \begin{array}{l}
		L(z_1,\theta_1,z_2,\theta_2) = R_1 z_1 + R_2 z_2 ,\\
		H(z_1,\theta_1,z_2,\theta_2) = (1 - t) z_1 + t\left(\sqrt{(1 - z_1^2)(1 - z_2^2)} \cos(\theta_1 - \theta_2) + z_1 z_2\right) .
	\end{array} \right.\]
	This chart covers the entire $M$ except for the four fixed points at the double poles and it is appropriate to study the system globally.
	
	On the other hand, to study the system around the fixed points at the double poles, we can use the same charts for $\mathbb{S}^2$ introduced in the study of the spin oscillator and write the system in double Cartesian coordinates.
	Explicitly, the appropriate charts to work everywhere except at the equators $z_1=0$ and $z_2=0$ of $M=\mathbb{S}^2 \times \mathbb{S}^2$ are $(\varphi_1,U_1^{\varepsilon_1})\times (\varphi_2,U_2^{\varepsilon_2})$, with $\varepsilon_1,\varepsilon_2\in\{+,-\}$ and where $U_i^{\varepsilon_i}=\{(x_i,y_i,z_i)\in\mathbb{S}^2\mid \varepsilon_i z_i>0\}\subset \mathbb{S}^2$ for $i=1,2$ and
	$$\begin{array}{rccc}
		\varphi_i:&\mathbb{S}^2\subset\mathbb{R}^3&\longrightarrow&\mathbb{R}^2\\
		&(x_i,y_i,z_i)&\longmapsto&(x_i,y_i)
	\end{array}.$$
	In these coordinates the symplectic form $\omega = -(R_1 \omega_{\mathbb{S}^2} + R_2 \omega_{\mathbb{S}^2})$ is
	$$\omega = -\varepsilon_1 R_1\frac{1}{\sqrt{1-x_1^2-y_1^2}}dx_1\wedge dy_1 - \varepsilon_2 R_2\frac{1}{\sqrt{1-x_2^2-y_2^2}}dx_2\wedge dy_2$$
	on $U_1^{\varepsilon_1}\times U_2^{\varepsilon_2}$, and the system writes as:
	\[\left\{ \begin{array}{l}
		L(x_1,y_1,x_2,y_2) = \varepsilon_1 R_1 \sqrt{1 - x_1^2 - y_1^2} + \varepsilon_2 R_2 \sqrt{1 - x_2^2 - y_2^2} \\
		H(x_1,y_1,x_2,y_2) = \varepsilon_1 (1-t) \sqrt{1 - x_1^2 - y_1^2} + t\left( x_1 x_2 + y_1 y_2 + \varepsilon_1 \varepsilon_2 \sqrt{(1 - x_1^2 - y_1^2) (1 - x_2^2 - y_2^2)} \right)
	\end{array} \right. .\]
	
	Observe that this system is an interpolation between a toric system when $t = 0$ and a semitoric system of toric type when $t=1$ (for the exact definition, see the work of V\~{u} Ng\d{o}c in \cite{VuNgoc2007}).
	In the first case $L$ is a coupled rotation and $H$ is a rotation, while in the second $L$ is again a coupled rotation and $H$ represents the angle.
	See the image of the momentum map of the system for different values of $0\leq t \leq 1$ in Figure \ref{fig:cam0}.
	
	\begin{figure}[hbt!]
		\centering
		\includegraphics[scale=0.17]{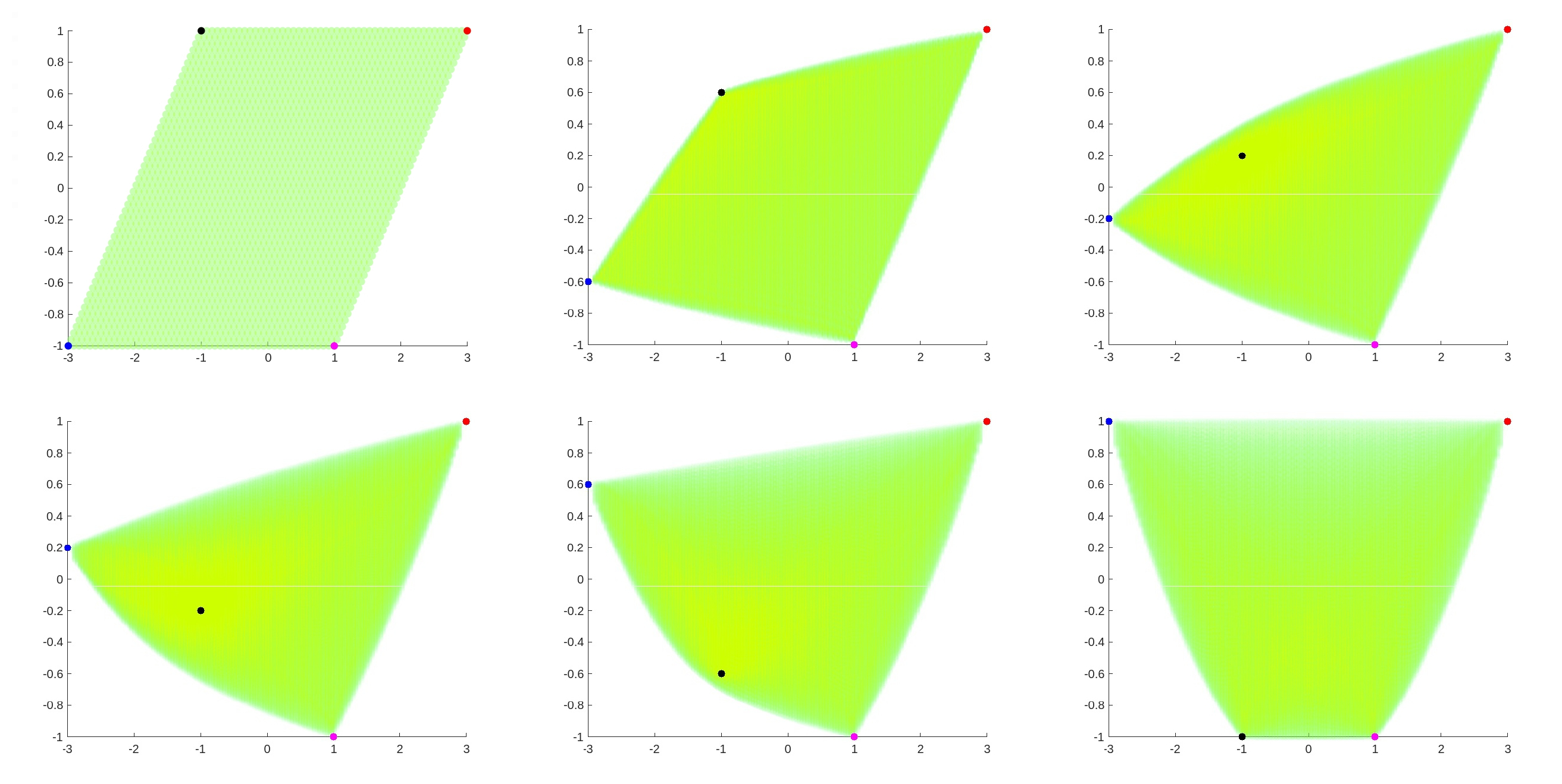}
		\caption{Image of the momentum map of the classical coupled angular momenta for values of $t$ between $0$ (top left) and $1$ (bottom right). The image of the focus-focus singularity is depicted in red.}
		\label{fig:cam0}
	\end{figure}
	
	\subsection{$b$-integrable systems}
	
	Let us start by recalling some definitions from Guillemin, Miranda, and Pires~\cite{GuilleminMirandaPires2014}.
	
	\begin{definition}
		A \emph{$b$-manifold} is a pair $(M,Z)$ where $M$ is an oriented smooth manifold and $Z$ is a closed and embedded submanifold of codimension $1$, commonly called the \emph{singular hypersurface}.
		A \emph{$b$-map} $f : (M,Z) \to (M',Z')$ is an orientation-preserving map $f : M \to M'$ such that $f^{-1}(Z') = Z$ and $f$ is transverse to $Z'$.
		A \emph{defining function} for $Z \subset M$ is a $b$-map $z : (M,Z) \to (\mathbb{R}, \{0\})$.
		
		The \emph{$b$-tangent bundle} of a $b$-manifold is the vector bundle $^bTM \to M$ whose sections are precisely the vector fields which are tangent to $Z$, also called the \emph{$b$-vector fields}.
		Its dual, denoted as $^bT^{\ast}M$, is the \emph{$b$-cotangent bundle}, and the sections of $\bigwedge^k \left({}^{b}T^{\ast}M\right)$ are denoted by $^b \Omega^k(M)$ and called the \emph{$b$-de Rham forms}.
		The restriction of any $b$-de Rham form to $M\setminus Z$ defines a smooth de Rham form there, and the differential $d : \Omega^k(M) \to \Omega^{k+1}(M)$ can be canonically extended to $^b\Omega^{\bullet}(M)$.
		
		A \emph{$b$-symplectic form} is a $b$-de Rham form of degree two $\omega \in {}^b\Omega^2(M)$ such that it is closed and non-degenerate.
	\end{definition}
	
	\begin{theorem}
		\cite[Theorem~27]{GuilleminMirandaPires2014}
		The $b$-cohomology groups of a $b$-manifold $(M,Z)$ are
		\[^bH^{\bullet}(M) \cong H^{\bullet}(M) \oplus H^{\bullet - 1}(Z) .\]
	\end{theorem}
	
	Locally around $p \in Z$ one can think that smooth forms are extended in $^b\Omega^k(M)$ by $b$-forms of the type $\frac{dz}z \wedge \eta$, with $\eta \in \Omega^{k-1}(M)$ and $z$ a local defining function for $Z$.
	In particular, the form $\frac{dz}z$ is always closed and not exact.
	
	
	\begin{definition}
		Let $(M,Z)$ be a $b$-manifold.
		We define the sheaf of $b$-functions $^b\mathcal{C}^{\infty}(M)$ by
		\[^b\mathcal{C}^{\infty}(U) = \left\{ c\log|z| + g \ | \ c \in \mathbb{R}, g \in \mathcal{C}^{\infty}(U), z \in \mathcal{C}^{\infty}(U) \text{ a local defining function of } Z \right\} .\]
		
		\label{def:bfunction}
	\end{definition}
	
	With this setting it is now possible for us to extend Definition \ref{def:integrablesystem} to $b$-manifolds.
	
	\begin{definition}
		\cite[Definition~57]{GuilleminMirandaPires2014}
		A \emph{$b$-integrable system} on a $b$-symplectic manifold $(M,Z,\omega)$ is a tuple $(f_1,..., f_n)$ of $b$-functions such that $\{f_i,f_j\} = 0$ for all $1 \leq i,j \leq n$ and such that $df_1\wedge ...\wedge df_n$ does not vanish (as a form in $^b\Omega^{2n}(M)$) almost everywhere in $M$ and also almost everywhere in $Z$.
		
		\label{def:bintegrablesystems}
	\end{definition}
	
	This definition has a remarkable difference with respect to Definition \ref{def:integrablesystem} besides the fact that we consider a $b$-symplectic form instead of a smooth symplectic form, which is that we require that $df_1,...,df_n$ to be independent almost everywhere on $Z$.
	This is chosen to avoid a situation in which $(f_1,...,f_n)$ reduces to a distribution of rank $2n-2$ on $Z$, which is too restrictive in order to prove normal form theorems (see for example Guillemin, Miranda, Pires and Scott~\cite{GuilleminMirandaPiresScott2015} or Kiesenhofer, Miranda and Scott~\cite{KiesenhoferMirandaScott2016}).
	
	As a consequence of the normal form of such a system (see for instance Kiesenhofer and Miranda~\cite[Remark~18]{KiesenhoferMirandaScott2016}), it is known that
	
	\begin{lemma}
		\label{lemma:nofixedbintegrable}
		
		Given a $b$-integrable system $(M,Z,\omega, F = (f_1,...,f_n))$ with non-degenerate singularities, there exist Eliasson-type normal forms in a neighbourhood of points in $Z$ and the minimal rank of $dF$ for these singularities is $1$ along $Z$.
		
		In particular, $Z$ cannot contain fixed points of the system.
	\end{lemma}
	
	\subsubsection{$b$-toric manifolds}
	
	The case of $b$-toric manifolds was thoroughly studied by Guillemin, Miranda, Pires and Scott in \cite{GuilleminMirandaPiresScott2015} and also by Gualtieri, Li, Pelayo and Ratiu in \cite{GualtieriLiPelayoRatiu2017}.
	Here we present a summary of the former's results.
	
	We denote by $\mathfrak{t}$ the Lie algebra of the torus $\mathbb{T}^n$ and by $X^{\#} \in \mathfrak{X}(M)$ the fundamental vector field associated to an element $X \in \mathfrak{t}$ by the action.
	
	\begin{definition}
		\cite[Definition~7]{GuilleminMirandaPiresScott2015}
		Let $(M,Z,\omega)$ a $b$-symplectic manifold, and consider a Lie group action by the torus $\mathbb{T}^n$.
		
		We say that it is \emph{Hamiltonian} if for all $X, Y \in \mathfrak{t}$:
		\begin{itemize}
			\item $\iota_{X^{\#}} \omega$ is exact, i.\ e.\, it has a primitive $H_X \in ^{b}\mathcal{C}^{\infty}(M)$.
			\item $\omega(X^{\#}, Y^{\#}) = 0$.
		\end{itemize}
		
		We say that it is \emph{toric} if it is effective and $\mathrm{dim}(\mathbb{T}^n) = \frac12 \mathrm{dim}(M)$.
		\label{def:tnactionb}
	\end{definition}
	
	Through an equivariant version of the $b$-Morse lemma it is possible to show a particularly simple classification of toric Hamiltonian $b$-actions in the particular case of surfaces.
	
	\begin{theorem}
		\cite[Theorem~9]{GuilleminMirandaPiresScott2015}
		A $b$-symplectic surface with a toric $\mathbb{S}^1$-action is equivariantly $b$-symplectomorphic to either $(\mathbb{S}^2,Z)$ or $(\mathbb{T}^2, Z)$, where $Z$ is a collection of latitude circles (in the $\mathbb{T}^2$ case, an even number of such circles), the action is the standard rotation, and the $b$-symplectic form is determined by the modular periods of the critical curves and the regularized Liouville volume.
		
		\label{theorem:classification2btoric}
	\end{theorem}
	
	The study of higher dimensional cases requires first an understanding of the behaviour of the $\mathbb{T}^n$-action semilocally near the hypersurface $Z$.
	To this end, an equivariant Darboux theorem is proved.
	Also, the authors introduce the notion of \emph{modular weight} of a connected component of $Z$.
	
	\begin{definition}
		\cite[Remark~10]{GuilleminMirandaPiresScott2015}
		For each connected component $Z' \subseteq Z$ there is an element $v_{Z'} \in \mathfrak{t}^{\ast} = \mathrm{Hom}(\mathfrak{t}, \mathbb{R})$, the \emph{modular weight} of $Z'$, such that for every $X \in \mathfrak{t}$ the function $H_X$ given by Definition \ref{def:tnactionb} has the form $v_{Z'}(X) \log|z| + g$ in a tubular neighbourhood around $Z'$, where $z$ is a local defining function of $Z'$ and $g \in \mathcal{C}^{\infty}(M)$.
		
		\label{def:modularweight}
	\end{definition}
	
	\begin{remark}
		\cite[Claim~13]{GuilleminMirandaPiresScott2015}
		If the action is toric, then $v_{Z'} \neq 0$.
	\end{remark}
	
	\begin{remark}
		\cite[Corollary~16]{GuilleminMirandaPiresScott2015}
		The hypersurface $Z$ is always a product, $Z \cong \mathcal{L}\times \mathbb{S}^1$.
	\end{remark}
	
	Using these tools it is possible to understand the global behaviour of a $b$-toric Hamiltonian action via an analogue to the Delzant polytope.
	In some sense, we want to understand the image of a momentum map, which as we see in Definition \ref{def:modularweight} is not a smooth function in a neighbourhood of $Z$.
	Away from $Z$, however, the following is true
	
	\begin{remark}
		\cite[Claim~19]{GuilleminMirandaPiresScott2015}
		For each connected component $W \subseteq M\setminus Z$, the image $\left.\mu\right|_W(W)$ is convex.
	\end{remark}
	
	To define globally the image of a $b$-momentum map it is necessary to provide a notion of its codomain.
	
	\begin{definition}
		\cite[Definition~21]{GuilleminMirandaPiresScott2015}
		Let $(M,Z,\omega)$ a $b$-symplectic manifold, and consider a toric Hamiltonian action by $\mathbb{T}^n$ on it.
		The \emph{adjacency graph} $\mathcal{G} = (G, w)$ associated to it consists of the graph $G = (V, E)$ whose vertices $v \in V$ are connected components of $M\setminus Z$ and has an edge between $v$ and $v'$ if there is a connected component of $Z$ that borders both $v$ and $v'$, and $w : E \to \mathfrak{t}^{\ast}$ is the map that associates to each connected component $Z$ its modular weight $v_Z$.
		When the action is effective, the graph $G$ must either be a cycle with an even number of vertices or a line.
	\end{definition}
	
	\begin{definition}
		Consider a pair $\mathcal{G} = (G, w)$ of such a graph and a function $w : E \to \mathfrak{t}^{\ast}$ such that $w(e) = k w(e')$ for $k < 0$ if $e$ and $e'$ meet at a vertex.
		The \emph{$b$-momentum codomain} $(\mathcal{R}_{\mathcal{G}}, \mathcal{Z}_{\mathcal{G}}, \hat{x})$ is a $b$-manifold $(\mathcal{R}_{\mathcal{G}}, \mathcal{Z}_{\mathcal{G}})$ with a smooth map $\hat{x} : \mathcal{R}_{\mathcal G} \setminus \mathcal{Z}_{\mathcal G} \to \mathfrak{t}^{\ast}$.
		A $b$-map $\mu : M \to \mathcal{R}_{\mathcal G}$ is then a \emph{momentum map} if it is $\mathbb{T}^n$ equivariant and $\mathfrak{t} \ni X \mapsto \mu^{X} \in \mathcal{C}^{\infty}(M)$ with $\mu^X(p) = \langle \hat{x} \circ \mu(p), X \rangle$ is linear, and moreover
		\[\iota_{X^{\#}}\omega = d\mu^X .\]
	\end{definition}
	
	For more information on how the codomain is defined see \cite[Section~5]{GuilleminMirandaPiresScott2015}.
	
	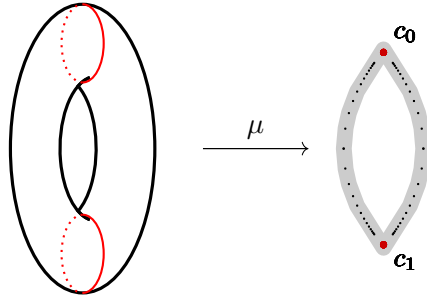
\begin{figure}[ht]
		\begin{tikzpicture}[scale = 0.8]
			
			\pgfmathsetmacro{\sizer}{2}
			\pgfmathsetmacro{\donutcenterx}{0}
			\pgfmathsetmacro{\donutcentery}{0}
			\pgfmathsetmacro{\basept}{5}

			\DrawDonut{\donutcenterx}{\donutcentery}{.6*\sizer}{1.2*\sizer}{-90}{black}{very thick}
			\draw [red, thick] (\donutcenterx, \donutcentery - .55*\sizer) arc (90:-90:0.35 and .65);
			\draw [red, thick, dotted] (\donutcenterx, \donutcentery - .55*\sizer) arc (90:270:0.35 and .65);
			
			\draw [red, thick] (\donutcenterx, \donutcentery + .55*\sizer) arc (-90:90:0.35 and .65);
			\draw [red, thick, dotted] (\donutcenterx, \donutcentery + .55*\sizer) arc (270:90:0.35 and .65);
			
			\foreach \paramt in {-3, -2.75, ..., 0.01}
			{
				\pgfmathsetmacro{\height}{1.5 * exp(\paramt)}
				
				\pgfmathsetmacro{\xshift}{0.25*(\height * (\height - 3)) - 0.1}
				\draw[fill = black] (\basept + \xshift, \height - 1.5) circle(.1mm);
				\draw[fill = black] (\basept - \xshift, \height - 1.5) circle(.1mm);
				\draw[red, fill = red] (\basept, -1.6) circle(.5mm) node[below right, black] {\small{$c_1$}};
				
				\draw[fill = black] (\basept + \xshift, 1.5 - \height) circle(.1mm);
				\draw[fill = black] (\basept - \xshift, 1.5 - \height) circle(.1mm);
				\draw[red, fill = red] (\basept, 1.6) circle(.5mm) node[above right, black] {\small{$c_0$}};
			}

			\draw (2, 0) edge node[above] {$\mu$} (3.75, 0);
			\draw[->] (2, 0) -- (3.75, 0);

			\draw[line width = 6pt, join = round, opacity=0.2] (\basept - .66, 0) -- (\basept - .63,  1.16 - 1.5) -- (\basept - .575,  .91 - 1.5) -- (\basept - .5, .71-1.5) -- (\basept - .44, .55-1.5) -- (\basept, -1.65) -- (\basept + .44, .55-1.5) -- (\basept + .5, .71-1.5) -- (\basept + .575,  .91 - 1.5) -- (\basept + .63,  1.16 - 1.5) -- (\basept + .66, 0) -- (\basept + .63,  1.5 - 1.16) -- (\basept + .575,  1.5 - .91) -- (\basept + .5, 1.5 - .71) -- (\basept + .44, 1.5 - .55) -- (\basept, 1.65) -- (\basept - .44, 1.5 - .55) -- (\basept - .5, 1.5-.71) -- (\basept - .575,  1.5-.91) -- (\basept - .63,  1.5-1.16) -- cycle;

		\end{tikzpicture}
		\caption{The moment map $\mu: \mathbb{T}^2 \rightarrow \mathcal{R}_{\mathcal{G}}$.}
		\label{fig:momentmapT2}
	\end{figure}
	
	\begin{definition}
		\cite[Definition~28]{GuilleminMirandaPiresScott2015}
		A \emph{$b$-symplectic toric manifold} is $(M^{2n}, Z, \omega, \mu : M \to \mathcal{R}_{\mathcal G})$, where $(M, Z, \omega)$ is $b$-symplectic and $\mu$ is a momentum map for some $b$-toric action on $(M,Z,\omega)$.
	\end{definition}
	
	\begin{definition}
		\cite[Definitions~30~and~32]{GuilleminMirandaPiresScott2015}
		A \emph{$b$-polytope} in $\mathcal{R}_{\mathcal G}$ is a bounded subset $P$ that intersects every component of $\mathcal{Z}_{\mathcal G}$ and can be expressed as a finite intersection of half-spaces.
		
		Such a polytope is \emph{Delzant} if
		\begin{itemize}
			\item In the case that $G$ is a line, if for every vertex $v \in P$ there is a lattice basis $\{u_i\}$ of $\mathfrak{t}^{\ast}$ such that the edges incident to $v$ can be written in a neighbourhood of $v$ as $v + tu_i$ for $t \geq 0$.
			\item In the case that $G$ is a cycle, if $\Delta_Z \subseteq \mathfrak{t}^{\ast}_{w}$ is Delzant.
		\end{itemize}
	\end{definition}
	
	With this we have all the notions required to establish a classification:
	
	\begin{theorem}
		\cite[Theorem~35]{GuilleminMirandaPiresScott2015}
		The map
		\[\left\{ \begin{array}{c} \text{$b$-symplectic toric manifolds} \\ (M, Z, \omega, \mu : M \to \mathcal{R}_{\mathcal G}) \end{array} \right\} \to \left\{ \begin{array}{c}  \text{Delzant $b$-polytopes} \\ \text{in $\mathcal{R}_{\mathcal G}$} \end{array} \right\} \]
		that sends a $b$-symplectic toric manifold to the image of its momentum map is a bijection, where $b$-symplectic toric manifolds are considered up to equivariant $b$-symplectomorphisms that preserve the momentum map.
		
		\label{theorem:bdelzant}
	\end{theorem}
	
	Theorem \ref{theorem:bdelzant} induces a particularly rigid classification of $b$-toric manifolds:
	
	\begin{corollary}
		\cite[Remark~38]{GuilleminMirandaPiresScott2015}
		Every $b$-toric manifold $b$-symplectomorphic to either
		\begin{itemize}
			\item A product of a $b$-symplectic $\mathbb{T}^2$ with a smooth toric manifold, or
			\item A manifold obtained from a product of a $b$-symplectic $\mathbb{S}^2$ with a smooth toric manifold by a sequence of symplectic cuts performed at the north and south “polar caps”, away from the critical hypersurface $Z$.
		\end{itemize}
	\end{corollary}
	
	\subsubsection{$b$-integrable systems with focus-focus singularities}
	
	
	In dimension $4$, toric systems are generalized to semitoric systems by allowing focus-focus singularities. The purpose of this article is to generalize $b$-toric systems in dimension $4$ into $b$-semitoric systems in the same way, i.e., allowing focus-focus singularities.
	
	The first examples of $b$-integrable systems admitting singular points with not only elliptic but also focus-focus components were developed by Kiesenhofer and Miranda~\cite{KiesenhoferMiranda17} in $6$-dimensional manifolds. In these examples, the singularities with a focus-focus component are located at the critical set of the $b$-symplectic structure and are obtained as a $b$-cotangent lift (see Kiesenhofer and Miranda~\cite{KiesenhoferMiranda17} for more details).
	
	\begin{example}
		Consider the group $G:= S^1 \times \mathbb{R}^+ \times S^1$ acting on $M:=S^1 \times \mathbb{R}^2$  in the following way:
		$$(\varphi, a, \alpha)\cdot (\theta, x_1,x_2):=(\theta + \varphi, a R_\alpha (x_1, x_2)),$$
		where $R_\alpha$ is the matrix corresponding to a rotation by an angle $\alpha$ in the $(x_1,x_2)$ plane.
		
		The twisted $b$-cotangent lift of this action induces a $b$-integrable system in the $b$-symplectic manifold $(T^*M,\omega=\frac{1}{p}dp\wedge d\theta + dy_1 \wedge dx_1 + dy_2 \wedge dx_2)$ with momentum map $F=(f_1,f_2,f_3)$ with:
		
		\begin{align*}
			f_1 &= \langle \lambda, X_1^\# \rangle =  \log|p|, \\
			f_2 &= \langle \lambda, X_2^\# \rangle =  x_1 y_1 + x_2 y_2, \\
			f_3 &= \langle \lambda, X_3^\# \rangle =  x_1 y_2 -y_1 x_2.
		\end{align*}
		The $f_2$ and $f_3$ components generate a family of singular points with a focus-focus component which are found at $x_1=x_2=y_1=y_2=0$.
	\end{example}
	
	\begin{definition}
		Let $(f_1,f_2,f_3)$ be a $b$-integrable system, and let $p\in M$ be a point where the system is singular. We say that the singularity is of \textbf{focus-focus type} if there is a local chart $(t,z, x_1, y_1,  x_2, y_2)$ centered at $p$ such that the critical hypersurface of $\omega$ is locally around $p$ given by $t=0$ and the integrable system is given by
		$$f_1 = c \log|t|, \quad  f_2 = x_1 y_1 + x_2 y_2,\quad f_3 = x_1 y_2 - y_1 x_2.$$
	\end{definition}
	
	
	In this article, we construct examples of $4$-dimensional $b$-integrable systems admitting focus-focus singularities. In Section \ref{sec:bsemitoric} we prove that these type of singular points can not be located at the critical set $Z$ where the $b$-symplectic form is singular. Then, in $4$-dimensional $b$-integrable systems, focus-focus singularities are only be found in $M\setminus Z$, which is an open symplectic manifold and we can apply the results on semitoric systems there.
	
	\section{$b$-semitoric systems}
	
	\label{sec:bsemitoric}
	
	
	\begin{definition}
		A $4$-dimensional $b$-integrable system $(M,Z,\omega, (L,H))$ is \emph{$b$-semitoric} if $L$ is proper and generates an effective $S^1$-action on $M$ and all singular points of $F = (L,H)$ are non-degenerate and do not include hyperbolic components.
		
		\label{def:bsemitoricsystems}
	\end{definition}
	
	As we have seen in Lemma \ref{lemma:nofixedbintegrable}, such a system cannot have fixed points in the critical hypersurface $Z$.
	There is, however, an additional perspective that can be pursued to recover the same result.
	
	\begin{proposition}
		\label{propo:nofixedbsemitoric}
		
		Let $(M,Z,\omega, (L,H))$ a $b$-semitoric system.
		Then the rank of $(dL, dH)$ at a point $p \in Z$ must be at least $1$.
	\end{proposition}
	
	\begin{proof}
		This proof is based on the idea laid out by Kiesenhofer and Miranda in \cite[Remark~33]{KiesenhoferMiranda17}
		We follow the notation from Bolsinov and Fomenko~\cite[Section 1.8]{BolsinovFomenko04}.
		
		Consider a point $p \in Z$ such that $\left.dL\right|_p = \left.dH\right|_p = 0$.
		Then, the linearization of the actions of the flows $\varphi_{X_L}^t$ and $\varphi_{X_H}^t$ generates an $\mathbb{R}^2$ action on $T_pM$ which by construction preserves the $b$-symplectic form $\omega$.
		This means that it induces a dimension 2 commutative Lie group $G(L,H) \subset \mathrm{Sp}(4,\mathbb{R})$, from which we can derive a commutative Lie subalgebra $K(L,H) \subset \mathfrak{sp}(4, \mathbb{R})$.
		
		However, the vector fields $X_L$ and $X_H$ are tangent to $Z$ at every point, and therefore $G(L,H)$ must preserve $T_p Z \subset T_p M$, a 3-dimensional subspace.
		Thus, the Lie algebra $K(L,H)$ must preserve $T_p Z$ as well.
		
		As the point $p$ is non-degenerate, we know that $K(L,H)$ must be a Cartan subalgebra, and therefore it must be conjugate to one of the matrix subalgebras from Equation \ref{eq:classification_cartan}.
		Of these algebras, only the ones that have hyperbolic components can possibly leave a 3-dimensional subspace invariant, which cannot be present in a $b$-semitoric system.
		Therefore, we conclude that the point $p$ must be degenerate, which also contradicts our hypothesis, so there cannot be such a point $p \in Z$.
	\end{proof}
	
	\begin{remark}
		In the proof of Proposition \ref{propo:nofixedbsemitoric} we did not make use of the condition that $(df_1, ..., df_n)$ has maximal rank almost everywhere in $Z$.
		This opens the door to the study of systems that are integrable in the sense of Definition \ref{def:integrablesystem} on $b$-symplectic manifolds and have hyperbolic singularities on $Z$.
	\end{remark}
	
	There are already examples proposed for the case of singularities of rank $1$ in Kiesenhofer and Miranda~\cite{KiesenhoferMiranda17}:
	
	\begin{example}
		Consider the group $G := S^1\times \mathbb{R}^{+}$ acting on $M := S^1\times \mathbb{R}$ in the following way:
		\[(\varphi, g) \cdot (\theta, x) := (\theta + \varphi, gx),\]
		i.e. on the $S^1$ component we have rotations and on the $R$ component we have multiplications.
		Then the Lie algebra basis $\left( \frac{\partial}{\partial \theta}, \frac{\partial}{\partial g} \right)$ induces the following fundamental vector fields on $M$:
		\[X_1 := \frac{\partial}{\partial \theta} , \ X_2 := x \frac{\partial}{\partial x} .\]
		We consider the twisted $b$-cotangent lift on $T^{\ast}M$, i.e.\ the $b$-symplectic structure $\omega = - d\lambda$ where
		\[\lambda := \log |p| d\theta + ydx\]
		and $(\theta, p, x, y)$ are the standard coordinates on $T^{\ast}M$.
		The lifted action on $T^{\ast}M$ is $b$-Hamiltonian with momentum map given by $\mu := (f_1, f_2)$:
		\[ \begin{array}{c} f_1 = \langle \lambda , X_1^{\#} \rangle = \log |p| , \\ f_2 = \langle \lambda , X_2^{\#} \rangle = xy .\end{array} \]
		The singularity point at $x=y=0$ of this $b$-integrable system is has a hyperbolic component.
	\end{example}
	
	
	
	
	
	
	With these phenomena in mind it makes sense to study $b$-semitoric systems where focus-focus singularities are present and away from the critical set $Z$.
	
	\section{The $b$-coupled spin-oscillator}
	\label{section:spin-oscillator}
	
	In this section we construct a $4$-dimensional $b$-integrable system with a non-degenerate singularity of focus-focus type away from the hypersurface $Z$ where the $b$-symplectic form is singular.
	This example arises from modifying the coupled spin-oscillator.
	
	\subsection{The $b$-coupled spin-oscillator}
	
	In this section we define the $b$-coupled spin-oscillator.
	It is built from the original coupled spin-oscillator (see Definition \ref{def:coupledspin-oscillator}) applying the change $z\mapsto\log\lvert z\rvert$ to both the symplectic form $\omega$ and the function $L$.
	
	Let $\rho_1,\rho_2>0$ be positive constants. Consider the $b$-manifold
	\[(M=\mathbb{S}^2\times\mathbb{R}^2, \ Z=\{(x,y,z)\in\mathbb{S}^2\mid z=0\}\times\mathbb{R}^2).\]
	Consider the $b$-symplectic form $\omega = -\rho_1\, \omega^b_{\mathbb{S}^2} + \rho_2\, \omega_{\mathbb{R}^2}$ on $M$, where $\omega^b_{\mathbb{S}^2}$ is the standard $b$-symplectic form on $(\mathbb{S}^2,Z=\{(x,y,z)\in\mathbb{S}^2\mid z=0\})$ and $\omega_{\mathbb{R}^2}$ is the standard symplectic form on $\mathbb{R}^2$.
	
	Recall from Definition \ref{def:coupledspin-oscillator} that we have coordinate charts on $U^{\pm}$ and on $U^0$.
	On $U^0$ and in coordinates $(\theta,z,u,v)$, the $b$-symplectic form can be written as
	$$\omega = -\rho_1\, d\theta \wedge \frac{dz}{z} + \rho_2\,du \wedge dv,$$
	while on $U^{\pm}$ and in coordinates $(x,y,u,v)$, it can be written as
	$$\omega = -\rho_1\frac{1}{1-x^2-y^2}dx\wedge dy + \rho_2\,du \wedge dv.$$
	
	\begin{definition}
		The \emph{$b$-coupled spin-oscillator} is a $4$-dimensional Hamiltonian integrable system $(M,\omega,F=(L,H))$, where
		\begin{equation}
			\begin{cases}
				L(x,y,z,u,v) = \rho_1 \log\lvert z \rvert + \frac{\rho_2}{2}\left(u^2+v^2\right)\\
				H(x,y,z,u,v) = \frac{1}{2}\left(xu+yv\right)
			\end{cases}.
			\label{eq:bcso}
		\end{equation}
		\label{def:bcoupledspin-oscillator}
	\end{definition}

	\begin{lemma}
		The $b$-coupled spin-oscillator is a $b$-integrable system.
		\label{lem:bcspintsys}
	\end{lemma}
	
	\begin{proof}
		Geometrically, $L$ is the momentum map for the simultaneous rotation of the sphere around its vertical axis and the plane around the origin, while $H$ measures the difference between the polar angle on the sphere and on the plane. Then, $H$ is constant along the flow of $L$ and the Poisson bracket $\{L,H\}$ vanishes everywhere. With explicit computations, in $M^{\pm}$:
		\begin{align*}
			\{L,H\} &= X_L(H) = \left(-y\frac{\partial}{\partial x} + x\frac{\partial}{\partial y} -v \frac{\partial}{\partial u} + u \frac{\partial}{\partial v}\right)\left(\frac{1}{2}(xu+yv)\right) = \frac{1}{2}\left(-yu+xv-vx+uy\right) = 0,
		\end{align*}
		and, in $M^0$:
		\begin{align*}
			\{L,H\} &= X_L(H) = \left(\frac{\partial}{\partial \theta} - v \frac{\partial}{\partial u} + u \frac{\partial}{\partial v}\right)\left(\frac{\sqrt{1-z^2}}{2}\left(u\cos\theta + v\sin\theta\right)\right)\\
			&= \frac{\sqrt{1-z^2}}{2}(-u\sin\theta + v\cos\theta-v\cos\theta+u\sin\theta) = 0.
		\end{align*}
	\end{proof}
	
	\begin{remark}
		Notice that the choice of signs in the $b$-symplectic form $\omega = -\rho_1\, \omega^b_{\mathbb{S}^2} + \rho_2\, \omega_{\mathbb{R}^2}$ is such that the flow of $L$, when projected to $\mathbb{S}^2$, turns around the vertical axis counterclockwise and, when projected to $\mathbb{R}^2$, also rotates counterclockwise. On the other hand, the constants $\rho_1,\rho_2>0$ in the definition of $\omega$ ensure that the flow of $L$ is $2\pi$-periodic. This way, using the embedding of $\mathbb{S}^2$ in $\mathbb{R}^3$ and projecting $\mathbb{S}^2$ to the $z = 0$ hyperplane, points of $\mathbb{S}^2$ and points of $\mathbb{R}^2$ move with the same angular velocity under the flow of $L$, making the scalar product $ux + vy = 2H$ constant and, thus, the system $\{L,H\}$ integrable.
	\end{remark}
	
	\begin{proposition}
		The only singularities of the $b$-coupled spin-oscillator are two non-degenerate fixed points of focus-focus type at the “north pole” $((0, 0, 1), (0, 0)) \in \mathbb{S}^2 \times \mathbb{R}^2$ and the “south pole” $((0, 0, -1), (0, 0)) \in \mathbb{S}^2 \times \mathbb{R}^2$.
		\label{prop:singularitiesbcso}
	\end{proposition}
	
	\begin{proof}
		A point in the $b$-coupled spin-oscillator is singular if the rank of $dF=(dL,dH)$ there is lower than $2$. It is equal to $0$ only at $x=y=u=v=0$ (or, equivalently, at $z=\pm 1,u=v=0$), where $dL$ and $dH$ vanish. Then, the integrable system $F=(L,H)$ has just two fixed points, the “north pole” $((0, 0, 1), (0, 0)) \in \mathbb{S}^2 \times \mathbb{R}^2$ and the “south pole” $((0, 0, -1), (0, 0)) \in \mathbb{S}^2 \times \mathbb{R}^2$. One can similarly check that the poles are the only points in $M$ where the $b$-Hamiltonian vector fields $X_L$ and $X_H$ vanish simultaneously.
		
		To prove that the poles are non-degenerate and of focus-focus type, we follow Section 1.8.2 of Bolsinov and Fomenko~\cite{BolsinovFomenko04}. In particular, we prove that the quadratic parts of $L$ and $H$, i.e., their Hessians $d^2L$ and $d^2H$, are independent as forms and that there exists a linear combination of the symplectic operators $\omega^{-1}d^2L$ and $\omega^{-1}d^2H$ with four different eigenvalues of the form $\pm a \pm ib$.
		
		At the north and south poles, in coordinates $(x,y,u,v)$, the Hessians of $L$ and $H$ and the matrix form of $\omega$ have the following expression:
		\begin{equation*}
			d^2L = 
			\begin{pmatrix}
				-\rho_1 & 0 & 0 & 0\\
				0 & -\rho_1 & 0 & 0\\
				0 & 0 & \rho_2 & 0\\
				0 & 0 & 0 & \rho_2
			\end{pmatrix}
			\hspace{20pt}
			d^2H= 
			\frac{1}{2}\begin{pmatrix}
				0 & 0 & 1 & 0\\
				0 & 0 & 0 & 1\\
				1 & 0 & 0 & 0\\
				0 & 1 & 0 & 0
			\end{pmatrix}
			\hspace{20pt}
			\Omega= 
			\begin{pmatrix}
				0 & -\rho_1 & 0 & 0\\
				\rho_1 & 0 & 0 & 0\\
				0 & 0 & 0 & \rho_2\\
				0 & 0 & -\rho_2 & 0
			\end{pmatrix}.
		\end{equation*}
		The matrices $d^2L$ and $d^2H$ are clearly independent and give raise to the following symplectic operators:
		\begin{equation*}
			A_L:=\Omega^{-1}d^2L = 
			\begin{pmatrix}
				0 & -1 & 0 & 0\\
				1 & 0 & 0 & 0\\
				0 & 0 & 0 & -1\\
				0 & 0 & 1 & 0
			\end{pmatrix}
			\hspace{40pt}
			A_H:=\Omega^{-1}d^2H= 
			\frac{1}{2}\begin{pmatrix}
				0 & 0 & 0 & \frac1{\rho_1} \\
				0 & 0 & \frac{-1}{\rho_1} & 0\\
				0 & \frac{-1}{\rho_2} & 0 & 0\\
				\frac1{\rho_2} & 0 & 0 & 0
			\end{pmatrix}.
		\end{equation*}
		The operator corresponding to the linear combination $A_L + 2A_H$ has the form
		\begin{equation*}
			\begin{pmatrix}
				0 & -1 & 0 & \frac1{\rho_1} \\
				1 & 0 & \frac{-1}{\rho_1} & 0\\
				0 & \frac{-1}{\rho_2} & 0 & -1\\
				\frac1{\rho_2} & 0 & 1 & 0
			\end{pmatrix},
		\end{equation*}
		and its four different complex eigenvalues are $\pm \frac{1}{\sqrt{\rho_1 \rho_2}} \pm i$, proving that the poles are non-degenerate singularities of focus-focus type.
		
		On the other hand, $F=(L,H)$ does not have any singular points where the rank of $dF$ is equal to $1$. Actually, the rank of $dF$ is equal to $0$ at the poles and equal to $2$ elsewhere, because the differentials of $L$ and $H$ are linearly independent on $M^0=M\setminus\{((0, 0, \pm 1), (0, 0))\}$. Indeed, on $M^0$ they write as:
		\begin{equation*}
			\begin{cases}
				dL = \rho_1 \frac{dz}{z} + \rho_2 \left(udu+vdv\right)\\
				dH = \frac{\sqrt{1-z^2}}{2}\left(\frac{-z^2}{1-z^2}\left(u\cos\theta + v\sin\theta\right) \frac{dz}{z}
				+ \left(-u\sin\theta + v\cos\theta\right) d\theta
				+ \cos\theta du
				+ \sin\theta dv\right)
			\end{cases}.
		\end{equation*}
		Observe that none of them vanishes and suppose that there is a point where the rank of $dF$ is equal to $1$. Then, $dL$ and $dH$ are linearly dependent at this point and there exists $\mu$ different from $0$ such that $\mu dL + dH = 0$ there. We have the following equivalences:
		\begin{align*}
			&\mu dL + dH = 0 \\
			\iff & \begin{cases}
				\frac{dz}{z}:\mu \rho_1 + \frac{-z^2}{1-z^2}\left(u\cos\theta + v\sin\theta\right) = 0\\
				d\theta:-u\sin\theta + v\cos\theta= 0\\
				du:\mu \rho_2 u + \cos\theta = 0\\
				dv:\mu \rho_2 v + \sin\theta = 0
			\end{cases} \\
			\iff & \begin{cases}
				\mu \rho_1 + \mu \rho_2\frac{z^2}{1-z^2}\left(u^2 + v^2\right) = 0\\
				\mu^2\rho_2^2(u^2+v^2)=1
			\end{cases} \\
			\iff & \mu^2 \rho_1 \rho_2 + \frac{z^2}{1-z^2} = 0,
		\end{align*}
		but the last equality is clearly never satisfied because $\mu^2 \rho_1 \rho_2$ is strictly positive while $\frac{z^2}{1-z^2}$ is always non-negative on $M^0$. Then, there is no point on $M^0$ where the rank of $dF$ is equal to $1$.
	\end{proof}
	
	\begin{proposition}
		The momentum map $(L,H) : S^2 \times \mathbb{R}^2 \to \mathbb{R}^2$ of the $b$-coupled spin-oscillator is surjective.
		\label{prop:mmbcsosurjective}
	\end{proposition}
	
	\begin{proof}
		We begin by claiming that $L : S^2 \times \mathbb{R}^2 \to \mathbb{R}$ is surjective.
		
		Indeed, it is clear that the equation $\rho_1 \log \lvert z\rvert + \frac{\rho_2}2 (u^2 + v^2) = \ell$ has a solution for any choice of $\rho_1, \rho_2 > 0$ and $\ell \in \mathbb{R}$: if $\ell = 0$, then $z=\pm 1$, $(u,v) = (0,0)$ is a preimage; if $\ell > 0$, then we can take $z=\pm 1$ and $(u,v)$ such that $u^2 + v^2 = \frac{2\ell}{\rho_2}$; and if $\ell < 0$ we can take $z = \pm \exp\left(\frac{\ell}{\rho_1}\right)$ and $(u,v) = (0,0)$.
		
		Further, we claim that $H$ is surjective when restricted to any given fiber $\{L = \ell\}$.
		
		For simplicity we restrict ourselves to the points $(x,y,z,u,v) \in S^2 \times \mathbb{R}^2$ such that the vectors $(x,y)$ and $(u,v)$ are collinear in $\mathbb{R}^2$.
		In that case, $H$ can be expressed as $H = \pm \frac12 \|(x,y)\| \|(u,v)\|$, where the sign depends on whether $(x,y)$ and $(u,v)$ point in the same or in opposite directions.
		Since $(x,y,z)$ lies in the sphere, we know that $\|(x,y)\| = \sqrt{1 - z^2}$.
		Let $N := \|(u,v)\|$.
		
		With these notations, the momentum map can be expressed in this restriction as
		\begin{equation*}
			\begin{cases}
				L(z,N) = \rho_1\log \lvert z \rvert + \frac{\rho_2}2 N^2 \\
				H(z,N) = \pm \frac12 \sqrt{1-z^2} N
			\end{cases}.
		\end{equation*}
		
		Let us now assume that $L(z,N) = \ell$ for some $\ell \in \mathbb{R}$.
		We will study separately the cases in which $\ell \geq 0$ and $\ell \leq 0$.
		
		\begin{itemize}
			\item If $\ell \geq 0$, then $z$ may take any value within $[-1,1]$, and we can isolate $N$ with respect to $z$,
			\[N = \sqrt{\frac2{\rho_2} \left(\ell - \rho_1 \log \lvert z \rvert\right)} ,\]
			which allows us to conclude that $N \geq \sqrt{\frac{2\ell}{\rho_2}}$.
			Moreover, the expression
			\[H_+ = \frac12 \sqrt{1 - z^2} N = \frac12 \sqrt{1 - z^2} \sqrt{\frac2{\rho_2} \left(\ell - \log |z|\right)}\]
			may take any non-negative value (as $H_+(1) = 0$, $\displaystyle\lim_{z\to 0} H_+(z) = +\infty$, and $H_+$ is continuous), and thus $H$ is surjective under the assumption that $L = \ell$.
			
			\item If $\ell \leq 0$, then $N$ may take any non-negative value, and we can isolate $|z|$ with respect to $N$,
			\[|z| = \exp\left(\frac1{\rho_1} \left(\ell - \frac{\rho_2}2 N^2\right)\right) ,\]
			which means that $|z| \leq \exp\left(\frac{\ell}{\rho_1}\right)$.
			We can conclude from this that the expression $H_+ = \frac12 \sqrt{1 - z^2} N$ may take any non-negative value, and therefore $H$ is surjective for the fiber $\{L = \ell\}$.
		\end{itemize}
		\begin{figure}[ht!]
			\begin{tikzpicture}
				\begin{axis}[ticks=none,
					axis y line=center,
					axis x line=middle,
					axis on top,xmin=-3,xmax=3,
					domain=-3:3,xlabel=$L$,ylabel=$H$]            \addplot[fill,mark=none,blue!10,samples=100]
					({x},{3}) \closedcycle;\addplot[fill,mark=none,blue!10,samples=100]
					({x},{-3}) \closedcycle;
					\addplot[red, mark=*,mark options={mark size=2pt}] coordinates {(0,0)};
				\end{axis}
			\end{tikzpicture}
			\caption{Image of the momentum map of the $b$-coupled spin-oscillator. The red dot is the image of the focus-focus singularities.}
			\label{fig:momentmapbcso}
		\end{figure}
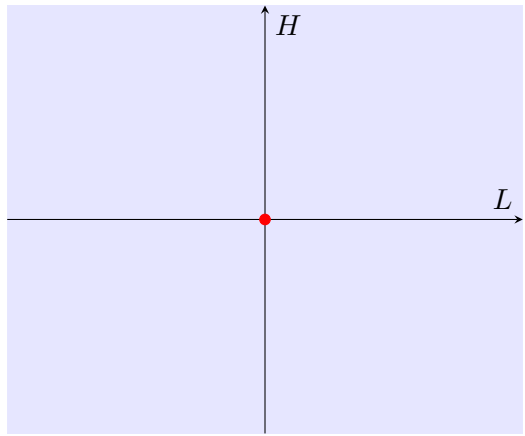
	\end{proof}
	
	\begin{remark}
		Notice that, due to the symmetry of the system, the image of a point $(x,y,u,v)$ by the momentum map $(L,H)$ coincides with the image of the point $(-x,-y,-u,-v)$. This, together with the fact that the image of $(L,H)$ is the whole $\mathbb{R}^2$, implies that the image of each open hemisphere is $\mathbb{R}^2$.
		\label{rem:symmetryofthesystem}
	\end{remark}
	
	\begin{corollary}
		The fact that the momentum map of the $b$-coupled spin-oscillator is surjective on both open hemispheres $M^+$ and $M^-$ can be also seen as a corollary of Proposition \ref{prop:singularitiesbcso}.
	\end{corollary}
	
	\begin{proof}
		Consider the Northern hemisphere $M^+$. It is a symplectic manifold and there the system is semitoric in the classical sense. The image of a focus-focus singularity in a semitoric system is in the interior of the whole image of the momentum map, so the origin of $\mathbb{R}^2$ (the image of the North pole) is contained in the interior of the image of $F=(L,H)$. Now, let $U\in\mathbb{R}^2$ be the set of points which are not in the image of $F$. Observe that $U$ is an open set and, then, if $U$ is non-empty, the image of $F$ has a non-empty boundary. In a semitoric system, the boundary of the image of the momentum map is made of points whose preimages are points on $M^+$ in which the momentum map is singular. But there are no other singular points in $M^+$ a part of the North pole, so $U$ has to be the empty set and $F$ is surjective on $M^+$.
		
		The same argument proves that $F$ is surjective on the Southern hemisphere $M^-$.
	\end{proof}
	
	\begin{remark}
		The results of Propositions \ref{prop:singularitiesbcso} and \ref{prop:mmbcsosurjective} are summarized in the image of the moment map drawn on Figure \ref{fig:momentmapbcso}. It shows that the $b$-coupled spin-oscillator behaves differently than the original coupled spin-oscillator, which has one focus-focus singularity, one elliptic-elliptic singularity and two one-parameter families of elliptic-regular singularities emanating from $p$ (see Figure \ref{fig:momentmapcso}).
	\end{remark}
	
	\subsection{The reversed $b$-coupled spin-oscillator}
	
	The signs of the $b$-symplectic form $\omega$ have to be consistent with the signs in the $L$ component of the momentum map (which is generating the $\mathbb{S}^1$-action) in order to define an integrable system. A choice was made in the definition of the $b$-coupled spin-oscillator (Definition \ref{def:bcoupledspin-oscillator}), but there is a second possibility which is also valid.
	
	Let $\rho_1,\rho_2>0$ be positive constants. Consider again the $b$-manifold
	\[(M=\mathbb{S}^2\times\mathbb{R}^2,\ Z=\{(x,y,z)\in\mathbb{S}^2\mid z=0\}\times\mathbb{R}^2).\]
	Consider now, differently, the $b$-symplectic form $\omega = \rho_1\, \omega^b_{\mathbb{S}^2} + \rho_2\, \omega_{\mathbb{R}^2}$ on $M$, where $\omega^b_{\mathbb{S}^2}$ is the standard $b$-symplectic form on $(\mathbb{S}^2,Z=\{(x,y,z)\in\mathbb{S}^2\mid z=0\})$ and $\omega_{\mathbb{R}^2}$ is the standard symplectic form on $\mathbb{R}^2$.
	
	Recall again the coordinate charts $U^0$ and $U^{\pm}$ from Definition \ref{def:coupledspin-oscillator}.
	On $U^0$ and in coordinates $(\theta,z,u,v)$, the $b$-symplectic form writes as
	$$\omega = \rho_1\, d\theta \wedge \frac{dz}{z} + \rho_2\,du \wedge dv,$$
	while on $U^{\pm}$ and in coordinates $(x,y,u,v)$, it writes as
	$$\omega = \rho_1\frac{1}{1-x^2-y^2}dx\wedge dy + \rho_2\,du \wedge dv.$$
	
	\begin{definition}
		The \emph{reversed $b$-coupled spin-oscillator} is a $4$-dimensional Hamiltonian integrable system $(M,\omega,F=(L,H))$, where
		\begin{equation}
			\begin{cases}
				L(x,y,z,u,v) = -\rho_1 \log\lvert z \rvert + \frac{\rho_2}{2}\left(u^2+v^2\right)\\
				H(x,y,z,u,v) = \frac{1}{2}\left(xu+yv\right)
			\end{cases}.
			\label{eq:bcsorev}
		\end{equation}
		\label{def:bcoupledspin-oscillatorrev}
	\end{definition}

	The reversed $b$-coupled spin-oscillator is essentially the same simultaneously rotating system as the $b$-coupled spin-oscillator but with a slight different coupling which does not change its integrability.
	
	\begin{lemma}
		The reversed $b$-coupled spin-oscillator is a $b$-integrable system.
		\label{lem:bcspintsysrev}
	\end{lemma}
	
	\begin{proof}
		Geometrically, we have the same picture as in the $b$-coupled spin-oscillator. Namely, $L$ is the momentum map for the simultaneous rotation of the sphere around its vertical axis and the plane around the origin, while $H$ measures the difference between the polar angle on the sphere and on the plane. Then, $H$ is constant along the flow of $L$ and the Poisson bracket $\{L,H\}$ vanishes everywhere.
	\end{proof}
	
	\begin{remark}
		The explicit computation of the Poisson bracket is analogous to the proof of Lemma \ref{lem:bcspintsys}, since both $X_L$ and $H$ have the same expression.
	\end{remark}
	
	In contrast with the $b$-coupled spin-oscillator, in the reversed system the two fixed points correspond to non-degenerate singularities of elliptic-elliptic type and there is an infinite number of singular points of rank $1$. More explicitly, we have:
	
	\begin{proposition}
		The singularities of the reversed $b$-coupled spin-oscillator are two non-degenerate fixed points of elliptic-elliptic type at the “north pole” $((0, 0, 1), (0, 0)) \in \mathbb{S}^2 \times \mathbb{R}^2$ and the “south pole” $((0, 0, -1), (0, 0)) \in \mathbb{S}^2 \times \mathbb{R}^2$ and four one-parameter families of elliptic-regular singularities emanating from the poles.
		\label{prop:singularitiesbcsorev}
	\end{proposition}
	
	\begin{proof}
		A point in the reversed $b$-coupled spin-oscillator is singular if the rank of $dF=(dL,dH)$ there is lower than $2$. Direct computation shows that it is $0$ only at the north and south poles $((0, 0, \pm 1), (0, 0)) \in \mathbb{S}^2 \times \mathbb{R}^2$. We follow again section 1.8.2 of Bolsinov and Fomenko~\cite{BolsinovFomenko04} to prove that the poles are non-degenerate fixed points of elliptic-elliptic type.
		
		At the north and south poles, in coordinates $(x,y,u,v)$, we have:
		\begin{equation*}
			d^2L = 
			\begin{pmatrix}
				\rho_1 & 0 & 0 & 0\\
				0 & \rho_1 & 0 & 0\\
				0 & 0 & \rho_2 & 0\\
				0 & 0 & 0 & \rho_2
			\end{pmatrix}
			\hspace{20pt}
			d^2H= 
			\frac{1}{2}\begin{pmatrix}
				0 & 0 & 1 & 0\\
				0 & 0 & 0 & 1\\
				1 & 0 & 0 & 0\\
				0 & 1 & 0 & 0
			\end{pmatrix}
			\hspace{20pt}
			\Omega= 
			\begin{pmatrix}
				0 & \rho_1 & 0 & 0\\
				-\rho_1 & 0 & 0 & 0\\
				0 & 0 & 0 & \rho_2\\
				0 & 0 & -\rho_2 & 0
			\end{pmatrix}.
		\end{equation*}
		The matrices $d^2L$ and $d^2H$ are independent and give raise to:
		\begin{equation*}
			A_L:=\Omega^{-1}d^2L = 
			\begin{pmatrix}
				0 & -1 & 0 & 0\\
				1 & 0 & 0 & 0\\
				0 & 0 & 0 & -1\\
				0 & 0 & 1 & 0
			\end{pmatrix}
			\hspace{40pt}
			A_H:=\Omega^{-1}d^2H= 
			\frac{1}{2}\begin{pmatrix}
				0 & 0 & 0 & \frac{-1}{\rho_1} \\
				0 & 0 & \frac1{\rho_1} & 0\\
				0 & \frac{-1}{\rho_2} & 0 & 0\\
				\frac1{\rho_2} & 0 & 0 & 0
			\end{pmatrix}.
		\end{equation*}
		For any $\gamma > 0$, the linear combination $A_L + 2 \gamma A_H$ has the form
		\begin{equation*}
			\begin{pmatrix}
				0 & -1 & 0 & \frac{-\gamma}{\rho_1} \\
				1 & 0 & \frac{\gamma}{\rho_1} & 0\\
				0 & \frac{-\gamma}{\rho_2} & 0 & -1\\
				\frac{\gamma}{\rho_2} & 0 & 1 & 0
			\end{pmatrix},
		\end{equation*}
		and has eigenvalues $\pm i\left(1+\frac{\gamma}{\sqrt{\rho_1 \rho_2}}\right),\pm i\left(1-\frac{\gamma}{\sqrt{\rho_1 \rho_2}}\right)$. Then, the linear combination $A_L + 2 \gamma A_H$ has four different imaginary eigenvalues of the type $\pm ia,\pm ib$ (except when $\gamma$ is exactly $\sqrt{\rho_1\rho_2}$, but we just need to show that there exists one linear combination of $A_L$ and $A_H$ with this property). This implies that the two poles are non-degenerate singularities of elliptic-elliptic type.
		
		Now, let us identify the singular points of $F=(L,H)$ on $M^0=M\setminus\{((0, 0, \pm 1), (0, 0))\}$ where the rank of $dF$ is equal to $1$. On $M^0$, the differentials of $L$ and $H$ write as:
		\begin{equation*}
			\begin{cases}
				dL = -\rho_1 \frac{dz}{z} + \rho_2 \left(udu+vdv\right)\\
				dH = \frac{\sqrt{1-z^2}}{2}\left(\frac{-z^2}{1-z^2}\left(u\cos\theta + v\sin\theta\right) \frac{dz}{z}
				+ \left(-u\sin\theta + v\cos\theta\right) d\theta
				+ \cos\theta du
				+ \sin\theta dv\right)
			\end{cases}.
		\end{equation*}
		None of them vanishes on $M^0$ and, then, at the points where the rank of $dF$ is equal to $1$, $dL$ and $dH$ are linearly dependent and there exists $\mu \neq 0$ such that $\mu dL + dH = 0$ there. Following the same computation that we did in the proof of Proposition \ref{prop:singularitiesbcso}, we have the following equivalences:
		\begin{equation*}
			\mu dL + dH = 0 \iff
			\begin{cases}
				-\mu \rho_1 + \frac{-z^2}{1-z^2}\left(u\cos\theta + v\sin\theta\right) = 0\\
				-u\sin\theta + v\cos\theta= 0\\
				\mu \rho_2 u + \cos\theta = 0\\
				\mu \rho_2 v + \sin\theta = 0
			\end{cases}
			\iff
			\begin{cases}
				\frac{z^2}{1-z^2} = \mu^2 \rho_1 \rho_2\\
				\mu^2\rho_2^2(u^2+v^2)=1\\
				u\sin\theta = v\cos\theta
			\end{cases}.
		\end{equation*}
		The last system can be solved for any value of $\mu\neq 0$, and has no solution with $z=0$ or with $(u,v)=(0,0)$. The space of solutions can be parametrized using just $2$ parameters. Explicitly, the set $K_1$ of singular points of rank $1$ on $M^0$ is a $2$-dimensional submanifold that can be parametrized by $\theta\in[0,2\pi)$ and $z\in(-1,0)\cup(0,1)$ as:
		\begin{equation*}
			(u(\theta,z),v(\theta,z))=\pm\sqrt{\frac{\rho_1}{\rho_2}}\frac{\sqrt{1-z^2}}{z}\left(\cos\theta,\sin\theta\right).
		\end{equation*}
		Observe that, for any $(\theta,z)\in [0,2\pi)\times (-1,0)\cup(0,1)\subset\mathbb{S}^2$, there are two points $(u,v)\in\mathbb{R}^2$ that solve the system of equations. If we look at the northern hemisphere, where $\theta\in[0,2\pi)$ and $z\in(0,1)$, there are two of them, both emanating from the respective poles. In the southern hemisphere it is analogous and this means that the submanifold of singular points of rank $1$ is made of four connected components.
		
		In the singular points of rank $1$ the $b$-Hamiltonian vector fields $X_L$ and $X_H$ are parallel and their flows generate $\mathbb{S}^1$-orbits. Since, in $M^0$, we have
		\begin{equation}
			X_L= \frac{\partial}{\partial \theta} - v \frac{\partial}{\partial u} + u \frac{\partial}{\partial v}.
		\end{equation}
		The $\mathbb{S}^1$-orbit of a singular point $(z,\theta,u,v)$ of rank $1$ consists of all the other singular points of rank $1$ that can be obtained from $(z,\theta,u,v)$ by a simultaneous rotation of $(z,\theta)$ around the vertical axis of $\mathbb{S}^2$ and of $(u,v)$ around the origin of $\mathbb{R}^2$. The four families of $\mathbb{S}^1$-orbits, in coordinates $(z,\theta,u,v),$ are the following:
		
		\begin{align*}
			&\left(z,\theta,\sqrt{\frac{\rho_1}{\rho_2}}\frac{\sqrt{1-z^2}}{z}\cos\theta,\sqrt{\frac{\rho_1}{\rho_2}}\frac{\sqrt{1-z^2}}{z}\sin\theta \right)\quad&\quad z\in(0,1),\theta\in [0,2\pi),\\
			&\left(z,\theta,-\sqrt{\frac{\rho_1}{\rho_2}}\frac{\sqrt{1-z^2}}{z}\cos\theta,-\sqrt{\frac{\rho_1}{\rho_2}}\frac{\sqrt{1-z^2}}{z}\sin\theta\right)\quad&\quad z\in(0,1),\theta\in [0,2\pi),\\
			&\left(z,\theta,\sqrt{\frac{\rho_1}{\rho_2}}\frac{\sqrt{1-z^2}}{z}\cos\theta,\sqrt{\frac{\rho_1}{\rho_2}}\frac{\sqrt{1-z^2}}{z}\sin\theta\right)\quad&\quad z\in(-1,0),\theta\in [0,2\pi),\\
			&\left(z,\theta,-\sqrt{\frac{\rho_1}{\rho_2}}\frac{\sqrt{1-z^2}}{z}\cos\theta,-\sqrt{\frac{\rho_1}{\rho_2}}\frac{\sqrt{1-z^2}}{z}\sin\theta\right)\quad&\quad z\in(-1,0),\theta\in [0,2\pi).
		\end{align*}

	\end{proof}
	
	\begin{corollary}
		The image of the momentum map $(L,H) : S^2 \times \mathbb{R}^2 \to \mathbb{R}^2$ of the reversed spin-oscillator is the open region of $\mathbb{R}^2$ bounded by the origin and the two open lines emanating from there and corresponding to the images of the elliptic-regular singularities.
	\end{corollary}
	
	\begin{proof}
		In the critical points of rank $1$ of $F$ the functions $L$ and $H$ on $M^0$ are:
		\begin{equation}
			\begin{cases}
				L(z,\theta,u,v)= -\rho_1 \log\lvert z \rvert + \rho_1 \frac{1-z^2}{2z^2}\\
				H(z,\theta,u,v)=\pm\sqrt{\frac{\rho_1}{\rho_2}}\frac{1-z^2}{2z}
			\end{cases}.
		\end{equation}
		As expected, the value of the momentum map in these points does not depend on the value of $\theta$ because it is the same in all the points of an $\mathbb{S}^1$-orbit.
		
		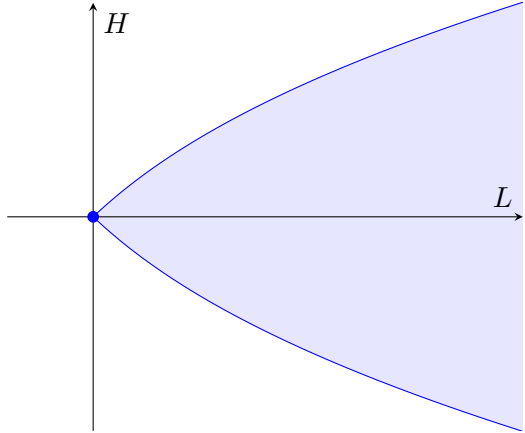
\begin{figure}[ht!]
			\begin{tikzpicture}
				\begin{axis}[ticks=none,
					axis y line=center,
					axis x line=middle,
					axis on top,xmin=-1,xmax=5,
					domain=0.2:1,xlabel=$L$,ylabel=$H$]
					\addplot[fill,mark=none,blue!10,samples=100]
					({-ln(x)+ (1-x^2)/(2*x^2)},{(1-x^2)/(2*x)}) \closedcycle;
					\addplot[mark=none,blue,samples=100]
					({-ln(x)+ (1-x^2)/(2*x^2)},{(1-x^2)/(2*x)});
					\addplot[fill,mark=none,blue!10,samples=100]
					({-ln(x)+ (1-x^2)/(2*x^2)},{-(1-x^2)/(2*x)}) \closedcycle;
					\addplot[mark=none,blue,samples=100]
					({-ln(x)+ (1-x^2)/(2*x^2)},{-(1-x^2)/(2*x)});
					\addplot[blue, mark=*,mark options={mark size=2pt}] coordinates {(0,0)};
				\end{axis}
			\end{tikzpicture}
			\caption{Image of the momentum map of the reversed $b$-coupled spin-oscillator. The blue dot is the image of the elliptic-elliptic singularities.}
			\label{fig:momentmapreversed}
		\end{figure}
		
	\end{proof}
	
	\begin{remark}
		The $b$-coupled spin-oscillator and the reversed $b$-coupled spin-oscillator are the two natural $b$-symplectic extensions of the classical coupled spin-oscillator formulated in Definition \ref{def:coupledspin-oscillator}.
	\end{remark}
	
	\section{The $b$-coupled angular momenta}
	\label{section:cam}
	
	There are three natural ways to construct a $b$-integrable system as a modification of the coupled angular momenta system $(L,H)$ defined on $(M=\mathbb{S}^2 \times \mathbb{S}^2,\omega=-(R_1 \omega_{\mathbb{S}^2} + R_2 \omega_{\mathbb{S}^2}))$ as seen in Section \ref{subsec:classicalcam}.
	All of them are obtained by selecting a hypersurface $Z\subset M$ and a $b$-symplectic structure on $M$ which is singular at $Z$, and by modifying the functions $L,H$ into $b$-functions compatible with the $b$-symplectic structure:
	
	\begin{itemize}
		\item System 1: We take $Z = \{z_1 = 0\}\subset M$, endow $M$ with the $b$-symplectic form $\omega=-(R_1 \frac{1}{z_1}\omega_{\mathbb{S}^2} + R_2 \omega_{\mathbb{S}^2})$ and modify $L$ to: $L(z_1,\theta_1,z_2,\theta_2) = R_1 \log|z_1| + R_2 z_2$.
		\item System 2: In the same setting as in the previous case, we also modify $H$ to: $H(z_1,\theta_1,z_2,\theta_2) = (1-t) \log|z_1| + t\left(\sqrt{(1 - z_1^2)(1 - z_2^2)} \cos(\theta_1 - \theta_2) + z_1 z_2\right)$.
		\item System 3: We take $Z = \{z_2 = 0\}\subset M$, endow $M$ with the $b$-symplectic form $\omega=-(R_1\omega_{\mathbb{S}^2} +  R_2 \frac{1}{z_2}\omega_{\mathbb{S}^2})$ and modify $L$ to: $L(z_1,\theta_1,z_2,\theta_2) = R_1 z_1 + R_2 \log|z_2|$.
	\end{itemize}
	
	Our goal in this section is to study these three $b$-integrable systems and classify their singularities. We start proving that all of them are indeed integrable and then we analyze separately for the three of them their singularities.
	
	\begin{lemma}
		Systems 1, 2 and 3 are $b$-integrable systems.
		\label{lem:camintsys}
	\end{lemma}
	
	\begin{proof}
		We check that the differentials $dL$ and $dH$ are independent almost everywhere and that $\{L,H\}=0$ on the whole $M$. We do it explicitly for System 1, since for System 2 and System 3 the computations can be reproduced analogously and yield the same results.
		
		First, in the double cylindrical charts $(z_1,\theta_1,z_2,\theta_2)\in U_1^0\times U_2^0$ introduced in Section \ref{subsec:classicalcam} $dL$ and $dH$ of System 1 are expressed as:
		\begin{equation}
			\left\{ \begin{array}{ll}
				dL(z_1,\theta_1,z_2,\theta_2) &= R_1 \frac{1}{z_1}dz_1 + R_2 dz_2 \\
				dH(z_1,\theta_1,z_2,\theta_2) &= \left(1 - t + tz_2 - t\frac{z_1}{\sqrt{1 - z_1^2}}\sqrt{1 - z_2^2} \cos(\theta_1 - \theta_2) \right) dz_1 \\
				& + t \left(z_1 -\frac{z_2}{\sqrt{1 - z_2^2}}\sqrt{1 - z_1^2} \cos(\theta_1 - \theta_2)\right)dz_2 \\
				& -t\sqrt{(1 - z_1^2)(1 - z_2^2)} \sin(\theta_1 - \theta_2) d\theta_1 \\
				& +t\sqrt{(1 - z_1^2)(1 - z_2^2)} \sin(\theta_1 - \theta_2) d\theta_2
			\end{array} \right. .
			\label{eq:case1diffscyl}
		\end{equation}
		
		The set $K\subset M$ of points where $dL$ and $dH$ are dependent is the union of three subsets, all of them of measure zero. The first one is just made by the four double poles, the only points of the system in which the differentials vanish simultaneously, as it can be checked in Cartesian coordinates. The second one consists of the points solving $\mu dL + \lambda dH = 0$ in Equation (\ref{eq:case1diffscyl}), which is a subset of $\{\theta_1=\theta_2\}\cup \{\theta_1=\theta_2+\pi\}$ and hence is a zero-measure set in $M$. The third one is made of a subset of the only part of $M$ not covered by the double Cartesian or the double cylindrical coordinates, i.e.\ the $2$-dimensional submanifolds $\{p_{+},p_{-}\}\times \mathbb{S}^2$ and $\mathbb{S}^2\times\{p_{+},p_{-}\}$, which is again of measure $0$. Therefore $dL\wedge dH \neq 0$ almost everywhere.

			Second, the computation of $\{L,H\}$ on $U_1^0\times U_2^0$ for System 1 yields:
			\begin{align*}
				\{L,H\} &= X_L(H) = \left(- \frac{\partial }{\partial \theta_1} - \frac{\partial }{\partial \theta_2}\right)\left((1 - t) z_1 + t\left(\sqrt{(1 - z_1^2)(1 - z_2^2)} \cos(\theta_1 - \theta_2) + z_1 z_2\right)\right) \\
				&= t\sqrt{(1 - z_1^2)(1 - z_2^2)}\sin(\theta_1 - \theta_2)- t\sqrt{(1 - z_1^2)(1 - z_2^2)}\sin(\theta_1 - \theta_2) = 0.
			\end{align*}
			By continuity of $\{L,H\}$, it is constantly $0$ on the entire manifold $M$.
		\end{proof}
		
		\begin{remark}
			For the sake of completeness, note that there is apparently another fourth equally natural option to change the classical coupled angular momenta into a $b$-integrable system, consisting of taking $Z$ as in Systems 1 and 2 and just modifying $H$ like in System 2, but this option does not really produce a $b$-integrable system. Indeed, if we take $Z = \{z_1 = 0\}$, $L(z_1,\theta_1,z_2,\theta_2) = R_1 z_1 + R_2 z_2$ and $H(x_1,y_1,z_1,x_2,y_2,z_2) = (1-t) \log|z_1| + t(x_1 x_2 + y_1 y_2 + z_1 z_2)$ then the system is not integrable.
			On an informal level we can make sense of this fact by taking into account that the vector field $X_L$ will no longer represent a rotation in this system.
			Explicitly, in $U_1^0\times U_2^0$,
			\begin{align*}
				\{L,H\} &= t (z_1 - 1) \sqrt{\left(1 - z_1^2\right) \left(1 - z_2^2\right)} \sin\left(\theta_1 - \theta_2\right) ,
			\end{align*}
			which is generically different from $0$.
		\end{remark}
		
		\subsection{Local analysis of the fixed points of the classical coupled angular momenta}
		\label{subsec:localcam}
		
		Before studying the critical points of the $3$ $b$-symplectic versions of the coupled angular momenta that we introduce in this section as $b$-integrable systems (cases (1), (2) and (3)), it will be useful to take a look at the local analysis of the already classified fixed points of the classical coupled angular momenta. 
		
		To determine the nature of the singular points of the $b$-symplectic systems, we will take advantage of the fact that the non-degeneracy and the type of a fixed point of an integrable system depend only on the local properties of the momentum map components up to second order (equivalently, at a fixed point of the system $(M,\omega,(H,L))$ the expressions $\omega$, $d^2L$ and $d^2H$ determine its type). More precisely, we will use that there exist similarities between the local forms at the fixed points of the $b$-integrable systems and the local forms at the fixed points of the classical coupled angular momenta. Since, as we mentioned, the latter ones are classified in function of the parameters $t,R_1,R_2$, using these similarities we will be able to characterize the former ones and to conclude about its type.
		
		For this purpose of analyzing the fixed points with the local expressions of the system, we include next the local expressions of $\Omega$, $d^2L$, $d^2H$, $\Omega^{-1}d^2L$, $\Omega^{-1}d^2H$ at the four fixed points $p_{\pm,\pm}$ of the classical coupled angular momenta. Later in this section we will compute the same local expressions at the fixed points of each of the $3$ $b$-integrable systems and we will find out that some of them coincide with these ones, meaning that the nature of the corresponding fixed points is the same.
		
		We do the computations for the four critical points $p_{\varepsilon_1,\varepsilon_2}$ simultaneously by using the combination of signs $\varepsilon_1,\varepsilon_2\in\{+,-\}$ corresponding to each fixed point. Namely, for each double pole $p_{\varepsilon_1,\varepsilon_2}=(0, 0, \varepsilon_1 1), (0, 0, \varepsilon_2 1)$ we use the Cartesian chart $(\varphi_1,U_1^{\varepsilon_1})\times (\varphi_2,U_2^{\varepsilon_2})$.
		
		At each double pole $p_{\varepsilon_1,\varepsilon_2}$, which corresponds to $(0,0,0,0)$ in coordinates $(x_1,y_1,x_2,y_2)$, we have:
		\begin{equation*}
			\Omega= 
			\begin{pmatrix}
				0 & -\varepsilon_1 R_1 & 0 & 0\\
				\varepsilon_1 R_1 & 0 & 0 & 0\\
				0 & 0 & 0 & -\varepsilon_2 R_2\\
				0 & 0 & \varepsilon_2 R_2 & 0
			\end{pmatrix},\quad\quad
			d^2L = 
			\begin{pmatrix}
				-\varepsilon_1 R_1 & 0 & 0 & 0\\
				0 & -\varepsilon_1 R_1 & 0 & 0\\
				0 & 0 & -\varepsilon_2 R_2 & 0\\
				0 & 0 & 0 & -\varepsilon_2 R_2
			\end{pmatrix},
		\end{equation*}
		\begin{equation*}
			d^2H= 
			\begin{pmatrix}
				\varepsilon_1(-1+t-\varepsilon_2 t) & 0 & t & 0\\
				0 & \varepsilon_1(-1+t-\varepsilon_2 t) & 0 & t\\
				t & 0 & -\varepsilon_1 \varepsilon_2 t & 0\\
				0 & t & 0 & -\varepsilon_1 \varepsilon_2 t
			\end{pmatrix}.
		\end{equation*}
		The matrices $d^2L$ and $d^2H$ are independent for any $t,\varepsilon_1,\varepsilon_2$ and give raise to $A_L^0:=\Omega^{-1}d^2L$ and $A_H^0:=\Omega^{-1}d^2H$, which have the expressions:
		\begin{equation*}
			A_L^0= 
			\begin{pmatrix}
				0 & -1 & 0 & 0\\
				1 & 0 & 0 & 0\\
				0 & 0 & 0 & -1\\
				0 & 0 & 1 & 0
			\end{pmatrix},\quad\quad
			A_H^0= 
			\begin{pmatrix}
				0 & \frac{-\varepsilon_2 t + t - 1}{R_1} & 0 & \frac{t}{\varepsilon_1 R_1} \\
				- \frac{-\varepsilon_2 t + t - 1}{R_1} & 0 & - \frac{t}{\varepsilon_1 R_1} & 0\\
				0 & \frac{t}{\varepsilon_2 R_2} & 0 & - \frac{\varepsilon_1 t}{R_2} \\
				-\frac{t}{\varepsilon_2 R_2} & 0 & \frac{\varepsilon_1 t}{R_2} & 0
			\end{pmatrix},
		\end{equation*}
		where the superscript $0$ just indicates that they correspond to the case of the classical coupled angular momenta. The linear combination $A^0:=A_L^0 + A_H^0$ has the form
		\begin{equation*}
			A^0=\begin{pmatrix}
				0 & \frac{-\varepsilon_2 t + t - 1}{R_1} -1 & 0 & \frac{t}{\varepsilon_1 R_1} \\
				- \frac{-\varepsilon_2 t + t - 1}{R_1} +1 & 0 & - \frac{t}{\varepsilon_1 R_1} & 0\\
				0 & \frac{t}{\varepsilon_2 R_2} & 0 & - \frac{\varepsilon_1 t}{R_2} -1\\
				-\frac{t}{\varepsilon_2 R_2} & 0 & \frac{\varepsilon_1 t}{R_2} +1 & 0
			\end{pmatrix}.
		\end{equation*}
		
		At each of the four poles $p_{\varepsilon_1,\varepsilon_2}$, $A^0$ is:
		\small
		\begin{equation*}
			A^0_{p_{+,+}}=\begin{pmatrix}
				0 & - \frac{1}{R_1} -1 & 0 & \frac{t}{R_1}\\
				\frac{1}{R_1} +1 & 0 & -\frac{t}{R_1} & 0\\
				0 & \frac{t}{R_2} & 0 & - \frac{t}{R_2} - 1\\
				-\frac{t}{R_2} & 0 & \frac{t}{R_2} + 1 & 0
			\end{pmatrix},
			A^0_{p_{+,-}}=\begin{pmatrix}
				0 & \frac{2t-1}{R_1} -1 & 0 & \frac{t}{R_1}\\
				-\frac{2t-1}{R_1} +1 & 0 & -\frac{t}{R_1} & 0\\
				0 & -\frac{t}{R_2} & 0 & - \frac{t}{R_2} -1\\
				\frac{t}{R_2} & 0 & \frac{t}{R_2} +1 & 0
			\end{pmatrix},
		\end{equation*}
		\begin{equation*}
			A^0_{p_{-,+}}=\begin{pmatrix}
				0 & -\frac{1}{R_1} -1 & 0 & -\frac{t}{R_1}\\
				\frac{1}{R_1} +1 & 0 & \frac{t}{R_1} & 0\\
				0 & \frac{t}{R_2} & 0 &  \frac{t}{R_2} -1\\
				-\frac{t}{R_2} & 0 & - \frac{t}{R_2} +1 & 0
			\end{pmatrix},
			A^0_{p_{-,-}}=\begin{pmatrix}
				0 & \frac{2t-1}{R_1} -1 & 0 & -\frac{t}{R_1}\\
				-\frac{2t-1}{R_1} +1 & 0 & \frac{t}{R_1} & 0\\
				0 & -\frac{t}{R_2} & 0 &  \frac{t}{R_2} -1\\
				\frac{t}{R_2} & 0 & - \frac{t}{R_2} +1 & 0
			\end{pmatrix}.
		\end{equation*}
		\normalsize
		And their characteristic polynomials are, respectively:
		
		\begin{align*}
			P^0_{+,+}(\lambda)=&\lambda^4 + \left( \frac{1}{R_1^2} + \frac{2(t^2 + R_2)}{R_1 R_2} + \frac{t (t + 2 R_2)}{R_2^2} + 2 \right)\lambda^2 + \frac{(-t^2 + t R_1 + t + R_1 R_2 + R_2)^2}{R_1^2 R_2^2},\\
			P^0_{+,-}(\lambda)=&\lambda^4 + \left( \frac{(1 - 2 t)^2}{R_1^2} + \frac{2(R_2 - t^2 - 2tR_2)}{R_1R_2} + \frac{t (t + 2 R_2)}{R_2^2} + 2 \right)\lambda^2 \\
			&+ \frac{(-t^2 + t R_1 - 2 t R_2 + t + R_1 R_2 + R_2)^2}{R_1^2 R_2^2},\\
			P^0_{-,+}(\lambda)=&\lambda^4 + \left(\frac{1}{R_1^2} + \frac{2(R_2 - t^2)}{R_1R_2} + \frac{t (t - 2 R_2)}{R_2^2} + 2 \right)\lambda^2 + \frac{(t^2 - t R_1 - t + R_1 R_2 + R_2)^2}{R_1^2 R_2^2},\\
			P^0_{-,-}(\lambda)=&\lambda^4 + \left( \frac{(1 - 2 t)^2}{R_1^2} + \frac{2 (R_2 + t^2 - 2 t R_2)}{R_1 R_2} + \frac{t (t - 2 R_2)}{R_2^2} + 2 \right)\lambda^2 \\
			&+ \frac{(-t^2 + t R_1 + 2 t R_2 + t - R_1 R_2 - R_2)^2}{R_1^2 R_2^2}.
		\end{align*}
		
		By the works on the classical coupled angular momenta (see Sadovksii and Zhilinskii~\cite{SadovskiiZhilinskii99}, Hohloch and Palmer~\cite{HohlochPalmer18}, Le Floch and Pelayo~\cite{LeFlochPelayo19}, Alonso, Dullin and Hohloch~\cite{AlonsoDullinHohloch20} and Alonso and Hohloch~\cite{AlonsoHohloch21}), the matrices $A^0_{p_{+,+}}, A^0_{p_{-,+}}, A^0_{p_{-,-}}$ have two pairs of imaginary eigenvalues and this shows that the non-degenerate fixed points $p_{+,+}, p_{-,+}, p_{-,-}$ are of elliptic-elliptic type. The matrix $A^0_{p_{+,-}}$ has two pairs of imaginary eigenvalues for $0\leq t<t^-$ or $1\geq t>t^+$ or four paired complex conjugate eigenvalues for $t^- < t < t^+$, which shows that $p_{+,-}$ is non-degenerate of elliptic-elliptic type or of focus-focus type, respectively. Also, for $t\in \{t^-,t^+\}$ $p_{+,-}$ is degenerate.
		
		\subsection{Critical points of System 1}
		We take $Z = \{z_1 = 0\}\subset M$, endow $M$ with the $b$-symplectic form $\omega=-(R_1 \frac{1}{z_1}\omega_{\mathbb{S}^2} + R_2 \omega_{\mathbb{S}^2})$ and modify $L$ to: $L(z_1,\theta_1,z_2,\theta_2) = R_1 \log|z_1| + R_2 z_2$.
		The expression of the $b$-symplectic form in cylindrical coordinates is
		\[\omega = R_1 \frac{dz_1}{z_1} \wedge d\theta_1 + R_2 dz_2 \wedge d\theta_2 ,\]
		and in Cartesian coordinates it is
		\[\omega = - \frac{R_1}{1 - x_1^2 - y_1^2} dx_1\wedge dy_1 - \varepsilon_2\frac{R_2}{\sqrt{1 - x_2^2 - y_2^2}} dx_2 \wedge dy_2 .\]
		The momenta in cylindrical coordinates now look like
		\[\left\{ \begin{array}{l} L(z_1,\theta_1,z_2,\theta_2) = R_1 \log|z_1| + R_2 z_2 \\ H(z_1,\theta_1,z_2,\theta_2) = (1 - t) z_1 + t\left(\sqrt{(1 - z_1^2)(1 - z_2^2)} \cos(\theta_1 - \theta_2) + z_1 z_2\right) \end{array} \right. ,\]
		and in Cartesian coordinates
		\[\left\{ \begin{array}{l} L(x_1,y_1,x_2,y_2) = \frac12 R_1 \log \left| 1 - x_1^2 - y_1^2 \right| + \varepsilon_2 R_2 \sqrt{1 - x_2^2 - y_2^2} \\ H(x_1,y_1,x_2,y_2) = \varepsilon_1(1-t) \sqrt{1 - x_1^2 - y_1^2} + t\left( x_1 x_2 + y_1 y_2 + \varepsilon_1\varepsilon_2\sqrt{(1 - x_1^2 - y_1^2) (1 - x_2^2 - y_2^2)} \right) \end{array} \right. .\]
		
		\begin{lemma}
			Let $\rho: \mathbb{S}^2 \times \mathbb{S}^2 \to \mathbb{S}^2 \times \mathbb{S}^2$ be the antipodal reflection in the first $\mathbb{S}^2$ component and the identity in the second $\mathbb{S}^2$ component.
			Then System 1 has the global symmetry $(L,H)=(L,-H)\circ\rho$.
			\label{lem:symmetrycase1}
		\end{lemma}
		
		\begin{proof}
			This global symmetry is expressed, in double cylindrical coordinates, as $(L,H)(z_1,\theta_1,z_2,\theta_2)=(L,-H)(-z_1,\theta_1+\pi,z_2,\theta_2)$. Direct computation shows that
			\begin{align*}
				L(-z_1,\theta_1+\pi,z_2,\theta_2)=R_1 \log|-z_1| + R_2 z_2=R_1 \log|z_1| + R_2 z_2=L(z_1,\theta_1,z_2,\theta_2)
			\end{align*}
			and that
			\begin{align*}
				H(-z_1,\theta_1+\pi,z_2,\theta_2)&=(1 - t) (-z_1) + t\left(\sqrt{(1 - (-z_1)^2)(1 - z_2^2)} \cos(\theta_1 +\pi - \theta_2) + (-z_1) z_2\right)\\
				&=-(1 - t) z_1 - t\left(\sqrt{(1 - z_1^2)(1 - z_2^2)} \cos(\theta_1 - \theta_2) + z_1 z_2\right)=-H(z_1,\theta_1,z_2,\theta_2)
			\end{align*}
			on $U_1^0\times U_2^0$. By continuity of $L$ and $H$, this equality extends to $M$.
		\end{proof}
		
		\begin{remark}
			The global symmetry is expressed as $(L,H)(x_1,y_1,x_2,y_2)=(L,-H)(-x_1,-y_1,x_2,y_2)$ in Cartesian coordinates. In this case, one has to take into account that if $(-x_1,-y_1,x_2,y_2)$ is covered by the chart $(\varphi_1,U_1^{\varepsilon_1})\times (\varphi_2,U_2^{\varepsilon_2})$, then $(x_1,y_1,x_2,y_2)$ is covered by the chart $(\varphi_1,U_1^{-\varepsilon_1})\times (\varphi_2,U_2^{\varepsilon_2})$.
		\end{remark}
		
		\begin{corollary}
			Two points of the system $(L,H)$ in $M=\mathbb{S}^2 \times \mathbb{S}^2$ which are antipodal in the first $\mathbb{S}^2$ component of $M$ and have the same coordinates in the second $\mathbb{S}^2$ component of $M$ have the same rank, non-degeneracy and type.
			\label{cor:symmetrycase1}
		\end{corollary}
		
		\begin{proof}
			The local normal form of an integrable system $(f_1,\dots,f_n)$ at a neighbourhood of a point is invariant under regular linear changes of the functions $f_1,\dots,f_n$ (see Bolsinov and Fomenko~\cite{BolsinovFomenko04}).
			Since the rank, the non-degeneracy and the type of a point are determined by its local normal form, these features corresponding to a point $p\in\mathbb{S}^2 \times \mathbb{S}^2$ coincide with these same features of the point $\rho(p)\in\mathbb{S}^2 \times \mathbb{S}^2$.
		\end{proof}
		
		\begin{proposition}
			System 1 has $4$ fixed points at the double poles $p_{\pm,\pm}=((0,0,\pm1),(0,0,\pm1))\in\mathbb{S}^2\times\mathbb{S}^2$. The fixed points $p_{+,+}$ and, $p_{-,+}$ are non-degenerate of elliptic-elliptic type for all values of $t$, while $p_{+,-}$ and $p_{-,-}$ are non-degenerate and of elliptic-elliptic type if $t<t^-$ or $t>t^+$, of focus-focus type if $t^- < t < t^+$ and degenerate if $t\in\{t^-,t^+\}$, where $$t^\pm = \frac{R_2}{2R_2+R_1\mp 2\sqrt{R_1 R_2}}.$$
			\label{prop:singularitiescam1}
		\end{proposition}
		
		\begin{proof}
			Singularities of the system correspond to points where the rank of $dF=(dL,dH)$ is lower than $2$. The rank of $dF$ is equal to $0$ only at the four double poles $p_{\pm,\pm}=((0, 0, \pm 1), (0, 0, \pm 1)) \in \mathbb{S}^2 \times \mathbb{S}^2$, meaning that they are the only fixed points of the system. To see that they are non-degenerate and to determine its type, we apply the local analysis in the same way that we did in Section \ref{subsec:localcam}. We do the computations for the four critical points $p_{\varepsilon_1,\varepsilon_2}$ simultaneously by using the combination of signs $\varepsilon_1,\varepsilon_2\in\{+,-\}$ corresponding to each double point, as we did in Section \ref{subsec:localcam}.
			
			At each double pole $p_{\varepsilon_1,\varepsilon_2}$, which corresponds to $(0,0,0,0)$ in coordinates $(x_1,y_1,x_2,y_2)$ in the chart $(\varphi_1,U_1^{\varepsilon_1})\times (\varphi_2,U_2^{\varepsilon_2})$, we have:
			\begin{equation*}
				\Omega= 
				\begin{pmatrix}
					0 & - R_1 & 0 & 0\\
					R_1 & 0 & 0 & 0\\
					0 & 0 & 0 & -\varepsilon_2 R_2\\
					0 & 0 & \varepsilon_2 R_2 & 0
				\end{pmatrix},\quad\quad
				d^2L = 
				\begin{pmatrix}
					- R_1 & 0 & 0 & 0\\
					0 & - R_1 & 0 & 0\\
					0 & 0 & -\varepsilon_2 R_2 & 0\\
					0 & 0 & 0 & -\varepsilon_2 R_2
				\end{pmatrix},
			\end{equation*}
			\begin{equation*}
				d^2H= 
				\begin{pmatrix}
					\varepsilon_1(-1+t-\varepsilon_2 t) & 0 & t & 0\\
					0 & \varepsilon_1(-1+t-\varepsilon_2 t) & 0 & t\\
					t & 0 & -\varepsilon_1 \varepsilon_2 t & 0\\
					0 & t & 0 & -\varepsilon_1 \varepsilon_2 t
				\end{pmatrix}.
			\end{equation*}
			The matrices $d^2L$ and $d^2H$ are independent for any $t,\varepsilon_1,\varepsilon_2$ and give raise to $A^1_L:=\Omega^{-1}d^2L$ and $A^1_H:=\Omega^{-1}d^2H$, which have the expressions:
			\begin{equation*}
				A^1_L = 
				\begin{pmatrix}
					0 & -1 & 0 & 0\\
					1 & 0 & 0 & 0\\
					0 & 0 & 0 & -1\\
					0 & 0 & 1 & 0
				\end{pmatrix},\quad\quad
				A^1_H = 
				\begin{pmatrix}
					0 & \frac{\varepsilon_1(-\varepsilon_2 t + t - 1)}{R_1} & 0 & \frac{t}{R_1} \\
					- \frac{\varepsilon_1(-\varepsilon_2 t + t - 1)}{R_1} & 0 & - \frac{t}{R_1} & 0\\
					0 & \frac{t}{\varepsilon_2 R_2} & 0 & - \frac{\varepsilon_1 t}{R_2} \\
					- \frac{t}{\varepsilon_2 R_2} & 0 & \frac{\varepsilon_1 t}{R_2} & 0
				\end{pmatrix}.
			\end{equation*}
			The linear combination $A^1:=A^1_L + A^1_H$ has the form
			\begin{equation*}
				A^1=\begin{pmatrix}
					0 & \frac{\varepsilon_1(-\varepsilon_2 t + t - 1)}{R_1} -1 & 0 & \frac{t}{R_1} \\
					- \frac{\varepsilon_1(-\varepsilon_2 t + t - 1)}{R_1} +1 & 0 & - \frac{t}{R_1} & 0\\
					0 & \frac{t}{\varepsilon_2 R_2} & 0 & - \frac{\varepsilon_1 t}{R_2} -1\\
					- \frac{t}{\varepsilon_2 R_2} & 0 & \frac{\varepsilon_1 t}{R_2} +1 & 0
				\end{pmatrix}.
			\end{equation*}
			
			At each of the four poles, $A^1$ is:
			
			\small
			\begin{equation*}
				A^1_{p_{+,+}}=\begin{pmatrix}
					0 & - \frac{1}{R_1} -1 & 0 & \frac{t}{R_1}\\
					\frac{1}{R_1} +1 & 0 & -\frac{t}{R_1} & 0\\
					0 & \frac{t}{R_2} & 0 & - \frac{t}{R_2} - 1\\
					-\frac{t}{R_2} & 0 & \frac{t}{R_2} + 1 & 0
				\end{pmatrix},
				A^1_{p_{+,-}}=\begin{pmatrix}
					0 & \frac{2t-1}{R_1} -1 & 0 & \frac{t}{R_1}\\
					-\frac{2t-1}{R_1} +1 & 0 & -\frac{t}{R_1} & 0\\
					0 & -\frac{t}{R_2} & 0 & - \frac{t}{R_2} -1\\
					\frac{t}{R_2} & 0 & \frac{t}{R_2} +1 & 0
				\end{pmatrix},
			\end{equation*}
			\begin{equation*}
				A^1_{p_{-,+}}=\begin{pmatrix}
					0 & \frac{1}{R_1} -1 & 0 & \frac{t}{R_1}\\
					-\frac{1}{R_1} +1 & 0 & -\frac{t}{R_1} & 0\\
					0 & \frac{t}{R_2} & 0 &  \frac{t}{R_2} -1\\
					-\frac{t}{R_2} & 0 & - \frac{t}{R_2} +1 & 0
				\end{pmatrix},
				A^1_{p_{-,-}}=\begin{pmatrix}
					0 & -\frac{2t-1}{R_1} -1 & 0 & \frac{t}{R_1}\\
					\frac{2t-1}{R_1} +1 & 0 & -\frac{t}{R_1} & 0\\
					0 & -\frac{t}{R_2} & 0 &  \frac{t}{R_2} -1\\
					\frac{t}{R_2} & 0 & - \frac{t}{R_2} +1 & 0
				\end{pmatrix}.
			\end{equation*}
			\normalsize
			
			Observe that $A^1_{p_{+,+}}$ is equal to $A^0_{p_{+,+}}$ from Section \ref{subsec:localcam}. Then, for all values of $t$, $p_{+,+}$ is also a non-degenerate fixed point of elliptic-elliptic type in this system. Similarly, since $A^1_{p_{+,-}}$ is equal to $A^0_{p_{+,-}}$, the fixed point $p_{+,-}$ is non-degenerate and of elliptic-elliptic type if $t<t^-$ or $t>t^+$, of focus-focus type for $t^- < t < t^+$ and degenerate for $t\in\{t^-,t^+\}$, where
			\begin{equation*}
				t^\pm = \frac{R_2}{2R_2+R_1\mp 2\sqrt{R_1 R_2}}.
			\end{equation*}
			
			On the other hand, by Corollary \ref{cor:symmetrycase1}, since $p_{-,+}$ is symmetric to $p_{+,+}$, it is of the same type, so it is a non-degenerate fixed point of elliptic-elliptic type. Similarly, $p_{-,-}$ is symmetric to $p_{+,-}$, so it is non-degenerate and of elliptic-elliptic type if $t<t^-$ or $t>t^+$, of focus-focus type if $t^- < t < t^+$ and degenerate if $t\in\{t^-,t^+\}$.
		\end{proof}
		
		
		\begin{figure}[ht!]
			\vspace{10mm}
			\centering
			\includegraphics[scale=0.17]{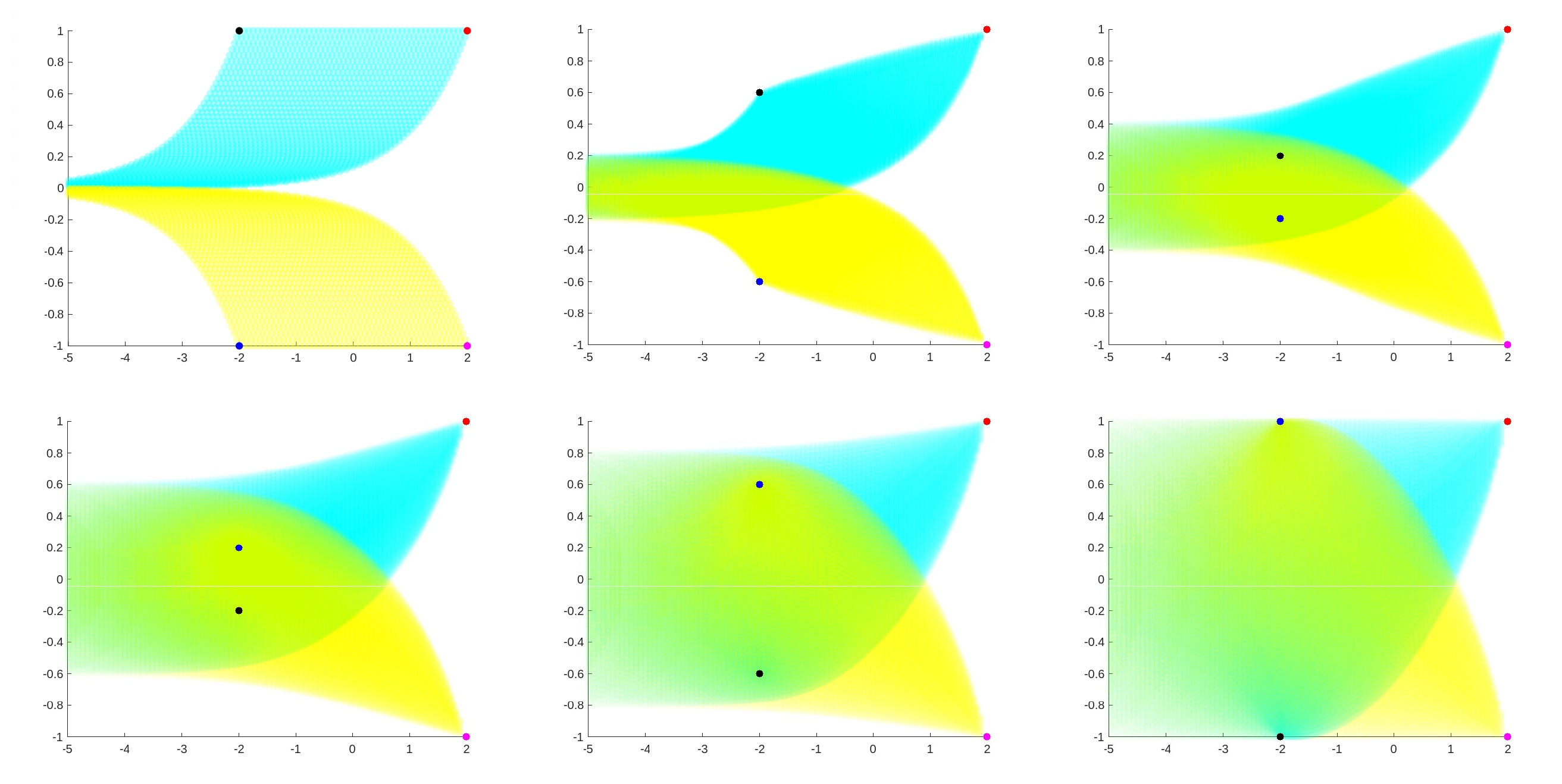}
			\caption{Image of the momentum map of System $1$ for values of $t$ between $0$ (top left) and $1$ (bottom right) by intervals of $0.2$. The images of the four fixed points $p_{+,+},p_{+,-},p_{-,+},p_{-,-}$ are depicted, respectively, in red, black, magenta and blue. The image of the submanifold $\mathbb{S}^+\times \mathbb{S}^2$ is depicted in sky blue and the image of the submanifold $\mathbb{S}^-\times \mathbb{S}^2$ is depicted in yellow, and both of them have been obtained numerically. Note that the image of the momentum map of a $b$-semitoric system is not necessarily convex (this also true for semitoric systems, see the example of the Hirzeburg surface in the work of Le Floch and Palmer~\cite{LeFlochPalmer2018}).}
			\label{fig:cam1}
		\end{figure}
		
		\subsection{Critical points of System 2}
		We take again $Z = \{z_1 = 0\}\subset M$, endow $M$ with the $b$-symplectic form $\omega=-(R_1 \frac{1}{z_1}\omega_{\mathbb{S}^2} + R_2 \omega_{\mathbb{S}^2})$, modify $L$ to: $L(z_1,\theta_1,z_2,\theta_2) = R_1 \log|z_1| + R_2 z_2$ and modify $H$ to: $H(z_1,\theta_1,z_2,\theta_2) = (1-t) \log|z_1| + t\left(\sqrt{(1 - z_1^2)(1 - z_2^2)} \cos(\theta_1 - \theta_2) + z_1 z_2\right)$.
		
		The expression of the $b$-symplectic form in cylindrical coordinates is
		\[\omega = R_1 \frac{dz_1}{z_1} \wedge d\theta_1 + R_2 dz_2 \wedge d\theta_2 ,\]
		and in Cartesian coordinates it is
		\[\omega = - \frac{R_1}{1 - x_1^2 - y_1^2} dx_1\wedge dy_1 - \varepsilon_2\frac{R_2}{\sqrt{1 - x_2^2 - y_2^2}} dx_2 \wedge dy_2 .\]
		
		The expressions of $(L,H)$ in cylindrical coordinates are
		\[\left\{ \begin{array}{l} L(z_1,\theta_1,z_2,\theta_2) = R_1 \log|z_1| + R_2 z_2 \\ H(z_1,\theta_1,z_2,\theta_2) = (1 - t) \log|z_1| + t\left(\sqrt{(1 - z_1^2)(1 - z_2^2)} \cos(\theta_1 - \theta_2) + z_1 z_2\right) \end{array} \right. ,\]
		and in Cartesian coordinates
		\[\left\{ \begin{array}{l} L(x_1,y_1,x_2,y_2) = \frac12 R_1 \log \left| 1 - x_1^2 - y_1^2 \right| + \varepsilon_2 R_2 \sqrt{1 - x_2^2 - y_2^2} \\ H(x_1,y_1,x_2,y_2) = (1-t) \frac12 \log{\left|1 - x_1^2 - y_1^2\right|} + t\left( x_1 x_2 + y_1 y_2 + \varepsilon_1\varepsilon_2 \sqrt{(1 - x_1^2 - y_1^2) (1 - x_2^2 - y_2^2)} \right) \end{array} \right. .\]
		
		\begin{proposition}
			System 2 has $4$ fixed points at the double poles $p_{\pm,\pm}=((0,0,\pm1),(0,0,\pm1))\in\mathbb{S}^2\times\mathbb{S}^2$. The fixed points $p_{+,+}$, $p_{-,+}$ and $p_{-,-}$ are non-degenerate of elliptic-elliptic type for all values of $t$, while $p_{+,-}$ is non-degenerate and of elliptic-elliptic type if $t<t^-$ or $t>t^+$, of focus-focus type if $t^- < t < t^+$ and degenerate if $t\in\{t^-,t^+\}$, where $$t^\pm = \frac{R_2}{2R_2+R_1\mp 2\sqrt{R_1 R_2}}.$$
			\label{prop:singularitiescam2}
		\end{proposition}
		
		\begin{proof}
			The only fixed points of the system are the four double poles $p_{\pm,\pm}=((0, 0, \pm 1), (0, 0, \pm 1)) \in \mathbb{S}^2 \times \mathbb{S}^2$. At each fixed point $p_{\varepsilon_1,\varepsilon_2}$, which corresponds to $(0,0,0,0)$ in coordinates $(x_1,y_1,x_2,y_2)$ in the chart $(\varphi_1,U_1^{\varepsilon_1})\times (\varphi_2,U_2^{\varepsilon_2})$, we have:
			\begin{equation*}
				\Omega= 
				\begin{pmatrix}
					0 & - R_1 & 0 & 0\\
					R_1 & 0 & 0 & 0\\
					0 & 0 & 0 & -\varepsilon_2 R_2\\
					0 & 0 & \varepsilon_2 R_2 & 0
				\end{pmatrix},\quad\quad
				d^2L = 
				\begin{pmatrix}
					- R_1 & 0 & 0 & 0\\
					0 & - R_1 & 0 & 0\\
					0 & 0 & -\varepsilon_2 R_2 & 0\\
					0 & 0 & 0 & -\varepsilon_2 R_2
				\end{pmatrix},
			\end{equation*}
			\begin{equation*}
				d^2H= 
				\begin{pmatrix}
					-1+t-\varepsilon_1\varepsilon_2 t & 0 & t & 0\\
					0 & -1+t-\varepsilon_1\varepsilon_2 t & 0 & t\\
					t & 0 & -\varepsilon_1 \varepsilon_2 t & 0\\
					0 & t & 0 & -\varepsilon_1 \varepsilon_2 t
				\end{pmatrix}.
			\end{equation*}
			The matrices $d^2L$ and $d^2H$ are independent for any $t,\varepsilon_1,\varepsilon_2$ and give raise to $A^2_L:=\Omega^{-1}d^2L$ and $A^2_H:=\Omega^{-1}d^2H$, which have the expressions:
			\begin{equation*}
				A^2_L= 
				\begin{pmatrix}
					0 & -1 & 0 & 0\\
					1 & 0 & 0 & 0\\
					0 & 0 & 0 & -1\\
					0 & 0 & 1 & 0
				\end{pmatrix},\quad\quad
				A^2_H= 
				\begin{pmatrix}
					0 & \frac{-\varepsilon_1\varepsilon_2 t + t - 1}{R_1} & 0 & \frac{t}{R_1} \\
					- \frac{-\varepsilon_1\varepsilon_2 t + t - 1}{R_1} & 0 & - \frac{t}{R_1} & 0\\
					0 & \frac{t}{\varepsilon_2 R_2} & 0 & - \frac{\varepsilon_1 t}{R_2} \\
					- \frac{t}{\varepsilon_2 R_2} & 0 & \frac{\varepsilon_1 t}{R_2} & 0
				\end{pmatrix}.
			\end{equation*}
			The linear combination $A^2:=A^2_L + A^2_H$ has the form
			\begin{equation*}
				A^2=\begin{pmatrix}
					0 & \frac{-\varepsilon_1\varepsilon_2 t + t - 1}{R_1} -1 & 0 & \frac{t}{R_1} \\
					- \frac{-\varepsilon_1\varepsilon_2 t + t - 1}{R_1} +1 & 0 & - \frac{t}{R_1} & 0\\
					0 & \frac{t}{\varepsilon_2 R_2} & 0 & - \frac{\varepsilon_1 t}{R_2} -1\\
					- \frac{t}{\varepsilon_2 R_2} & 0 & \frac{\varepsilon_1 t}{R_2} +1 & 0
				\end{pmatrix}.
			\end{equation*}
			
			At each of the four poles, $A^2$ is:
			\small
			\begin{equation*}
				A^2_{p_{+,+}}=\begin{pmatrix}
					0 & - \frac{1}{R_1} -1 & 0 & \frac{t}{R_1}\\
					\frac{1}{R_1} +1 & 0 & -\frac{t}{R_1} & 0\\
					0 & \frac{t}{R_2} & 0 & - \frac{t}{R_2} - 1\\
					-\frac{t}{R_2} & 0 & \frac{t}{R_2} + 1 & 0
				\end{pmatrix},
				A^2_{p_{+,-}}=\begin{pmatrix}
					0 & \frac{2t-1}{R_1} -1 & 0 & \frac{t}{R_1}\\
					-\frac{2t-1}{R_1} +1 & 0 & -\frac{t}{R_1} & 0\\
					0 & -\frac{t}{R_2} & 0 & - \frac{t}{R_2} -1\\
					\frac{t}{R_2} & 0 & \frac{t}{R_2} +1 & 0
				\end{pmatrix},
			\end{equation*}
			\begin{equation*}
				A^2_{p_{-,+}}=\begin{pmatrix}
					0 & \frac{2t-1}{R_1} -1 & 0 & \frac{t}{R_1}\\
					-\frac{2t-1}{R_1} +1 & 0 & -\frac{t}{R_1} & 0\\
					0 & \frac{t}{R_2} & 0 &  \frac{t}{R_2} -1\\
					-\frac{t}{R_2} & 0 & - \frac{t}{R_2} +1 & 0
				\end{pmatrix},
				A^2_{p_{-,-}}=\begin{pmatrix}
					0 & - \frac{1}{R_1} -1 & 0 & \frac{t}{R_1}\\
					\frac{1}{R_1} +1 & 0 & -\frac{t}{R_1} & 0\\
					0 & -\frac{t}{R_2} & 0 &  \frac{t}{R_2} -1\\
					\frac{t}{R_2} & 0 & - \frac{t}{R_2} +1 & 0
				\end{pmatrix}.
			\end{equation*}
			\normalsize
			
			Since $A^2_{p_{+,+}}$ is identical to $A^0_{p_{+,+}}$ from Section \ref{subsec:localcam}, $p_{+,+}$ is non-degenerate of elliptic-elliptic type for all values of $t$. Similarly, since $A^2_{p_{+,-}}$ is identical to $A^0_{p_{+,-}}$, $p_{+,-}$ is non-degenerate and of elliptic-elliptic type if $t<t^-$ or $t>t^+$, of focus-focus type if $t^- < t < t^+$ and degenerate if $t\in\{t^-,t^+\}$, where
			\begin{equation*}
				t^\pm = \frac{R_2}{2R_2+R_1\mp 2\sqrt{R_1 R_2}}.
			\end{equation*}
			On the other hand, by direct computation one can see that the characteristic polynomial of $A^2_{p_{-,+}}$ coincides with $P^0_{-,-}(\lambda)$ defined in Section \ref{subsec:localcam}. Then, $p_{-,+}$ is a non-degenerate fixed point of elliptic-elliptic type for all values of $t$. Similarly, the characteristic polynomial of $A^2_{p_{-,-}}$ coincides with $P^0_{-,+}(\lambda)$ of Section \ref{subsec:localcam}, and then $p_{-,-}$ is also a non-degenerate fixed point of elliptic-elliptic type for all values of $t$.
		\end{proof}
		
		\begin{figure}[ht!]
			\vspace{20mm}
			\centering
			\includegraphics[scale=0.17]{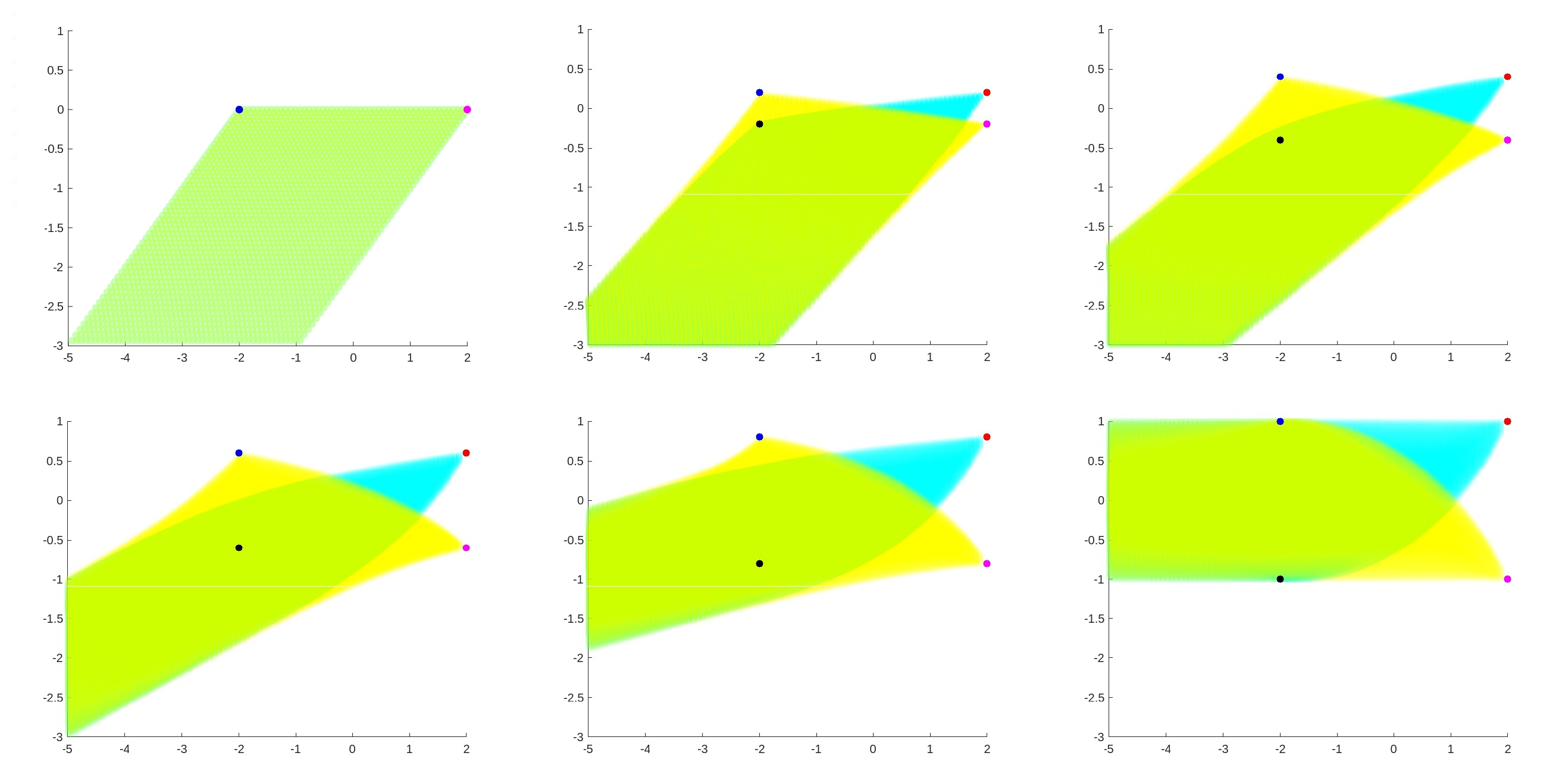}
			\caption{Image of the momentum map of System $2$ for values of $t$ between $0$ (top left) and $1$ (bottom right) by intervals of $0.2$. The images of the four fixed points $p_{+,+},p_{+,-},p_{-,+},p_{-,-}$ are depicted, respectively, in red, black, magenta and blue. The image of the submanifold $\mathbb{S}^+\times \mathbb{S}^2$ is depicted in sky blue and the image of the submanifold $\mathbb{S}^-\times \mathbb{S}^2$ is depicted in yellow, and both of them have been obtained numerically.}
			\label{fig:cam2}
		\end{figure}
		
		\subsection{Critical points of System 3}
		
		We take $Z = \{z_2 = 0\}\subset M$, endow $M$ with the $b$-symplectic form $\omega=-(R_1\omega_{\mathbb{S}^2} +  R_2 \frac{1}{z_2}\omega_{\mathbb{S}^2})$ and modify $L$ to: $L(z_1,\theta_1,z_2,\theta_2) = R_1 z_1 + R_2 \log|z_2|$. The $b$-symplectic form in cylindrical coordinates has the expression
		\[\omega = R_1 dz_1 \wedge d\theta_1 + R_2 \frac{dz_2}{z_2} \wedge d\theta_2,\]
		
		and in Cartesian coordinates it is
		\[\omega = - \frac{R_1}{\sqrt{1 - x_1^2 - y_1^2}} dx_1\wedge dy_1 - \frac{R_2}{1 - x_2^2 - y_2^2} dx_2 \wedge dy_2 .\]
		
		The expressions of $(L,H)$ in cylindrical coordinates are
		\[\left\{ \begin{array}{l} L(z_1,\theta_1,z_2,\theta_2) = R_1 z_1 + R_2 \log|z_2| \\ H(z_1,\theta_1,z_2,\theta_2) = (1 - t) z_1 + t\left(\sqrt{(1 - z_1^2)(1 - z_2^2)} \cos(\theta_1 - \theta_2) + z_1 z_2\right) \end{array} \right. ,\]
		and in Cartesian coordinates
		\[\left\{ \begin{array}{l} L(x_1,y_1,x_2,y_2) = R_1 \sqrt{1 - x_1^2 - y_1^2} + \frac12 R_2 \log \left| 1 - x_2^2 - y_2^2 \right| \\ H(x_1,y_1,x_2,y_2) = (1-t) \sqrt{1 - x_1^2 - y_1^2} + t\left( x_1 x_2 + y_1 y_2 + \sqrt{(1 - x_1^2 - y_1^2) (1 - x_2^2 - y_2^2)} \right) \end{array} \right. .\]
		
		\begin{proposition}
			System 3 has $4$ fixed points at the double poles $p_{\pm,\pm}=((0,0,\pm1),(0,0,\pm1))\in\mathbb{S}^2\times\mathbb{S}^2$. The fixed points $p_{+,+}$, $p_{+,-}$ and $p_{-,+}$ are non-degenerate of elliptic-elliptic type for all values of $t$, while $p_{-,-}$ is non-degenerate and of elliptic-elliptic type if $t<t^-$ or $t>t^+$, of focus-focus type if $t^- < t < t^+$ and degenerate if $t\in\{t^-,t^+\}$, where $$t^\pm = \frac{R_2}{2R_2+R_1\mp 2\sqrt{R_1 R_2}}.$$
			\label{prop:singularitiescam3}
		\end{proposition}
		
		\begin{proof}
			The only fixed points of the system are the four double poles $p_{\pm,\pm}=((0, 0, \pm 1), (0, 0, \pm 1)) \in \mathbb{S}^2 \times \mathbb{S}^2$. At each fixed point $p_{\varepsilon_1,\varepsilon_2}$, which corresponds to $(0,0,0,0)$ in coordinates $(x_1,y_1,x_2,y_2)$ in the chart $(\varphi_1,U_1^{\varepsilon_1})\times (\varphi_2,U_2^{\varepsilon_2})$, we have:
			\begin{equation*}
				\Omega= 
				\begin{pmatrix}
					0 & -\varepsilon_1 R_1 & 0 & 0\\
					\varepsilon_1 R_1 & 0 & 0 & 0\\
					0 & 0 & 0 & - R_2\\
					0 & 0 & R_2 & 0
				\end{pmatrix},\quad\quad
				d^2L = 
				\begin{pmatrix}
					-\varepsilon_1 R_1 & 0 & 0 & 0\\
					0 & -\varepsilon_1 R_1 & 0 & 0\\
					0 & 0 & - R_2 & 0\\
					0 & 0 & 0 & - R_2
				\end{pmatrix},
			\end{equation*}
			\begin{equation*}
				d^2H= 
				\begin{pmatrix}
					\varepsilon_1(-1+t-\varepsilon_2 t) & 0 & t & 0\\
					0 & \varepsilon_1(-1+t-\varepsilon_2 t) & 0 & t\\
					t & 0 & -\varepsilon_1 \varepsilon_2 t & 0\\
					0 & t & 0 & -\varepsilon_1 \varepsilon_2 t
				\end{pmatrix}.
			\end{equation*}
			The matrices $d^2L$ and $d^2H$ are independent for any $t,\varepsilon_1,\varepsilon_2$ and give raise to $A^3_L:=\Omega^{-1}d^2L$ and $A^3_H:=\Omega^{-1}d^2H$, which have the expressions:
			\begin{equation*}
				A^3_L= 
				\begin{pmatrix}
					0 & -1 & 0 & 0\\
					1 & 0 & 0 & 0\\
					0 & 0 & 0 & -1\\
					0 & 0 & 1 & 0
				\end{pmatrix},\quad\quad
				A_H^3= 
				\begin{pmatrix}
					0 & \frac{-\varepsilon_2 t + t - 1}{R_1} & 0 & \frac{t}{\varepsilon_1 R_1} \\
					- \frac{-\varepsilon_2 t + t - 1}{R_1} & 0 & - \frac{t}{\varepsilon_1 R_1} & 0\\
					0 & \frac{t}{R_2} & 0 & - \frac{\varepsilon_1\varepsilon_2 t}{R_2} \\
					-\frac{t}{R_2} & 0 & \frac{\varepsilon_1\varepsilon_2 t}{R_2} & 0
				\end{pmatrix}.
			\end{equation*}
			The linear combination $A^3:=A_L^3 + A_H^3$ has the form
			\begin{equation*}
				A^3=\begin{pmatrix}
					0 & \frac{-\varepsilon_2 t + t - 1}{R_1} -1 & 0 & \frac{t}{\varepsilon_1 R_1} \\
					- \frac{-\varepsilon_2 t + t - 1}{R_1} +1 & 0 & - \frac{t}{\varepsilon_1 R_1} & 0\\
					0 & \frac{t}{R_2} & 0 & - \frac{\varepsilon_1\varepsilon_2 t}{R_2} -1\\
					-\frac{t}{R_2} & 0 & \frac{\varepsilon_1\varepsilon_2 t}{R_2} +1 & 0
				\end{pmatrix}.
			\end{equation*}
			
			At each of the four poles $p_{\varepsilon_1,\varepsilon_2}$, $A^3$ is:
			\small
			\begin{equation*}
				A^3_{p_{+,+}}=\begin{pmatrix}
					0 & - \frac{1}{R_1} -1 & 0 & \frac{t}{R_1}\\
					\frac{1}{R_1} +1 & 0 & -\frac{t}{R_1} & 0\\
					0 & \frac{t}{R_2} & 0 & - \frac{t}{R_2} - 1\\
					-\frac{t}{R_2} & 0 & \frac{t}{R_2} + 1 & 0
				\end{pmatrix},
				A^3_{p_{+,-}}=\begin{pmatrix}
					0 & \frac{2t-1}{R_1} -1 & 0 & \frac{t}{R_1}\\
					-\frac{2t-1}{R_1} +1 & 0 & -\frac{t}{R_1} & 0\\
					0 & \frac{t}{R_2} & 0 & \frac{t}{R_2} -1\\
					-\frac{t}{R_2} & 0 & - \frac{t}{R_2} +1 & 0
				\end{pmatrix},
			\end{equation*}
			\begin{equation*}
				A^3_{p_{-,+}}=\begin{pmatrix}
					0 & -\frac{1}{R_1} -1 & 0 & -\frac{t}{R_1}\\
					\frac{1}{R_1} +1 & 0 & \frac{t}{R_1} & 0\\
					0 & \frac{t}{R_2} & 0 &  \frac{t}{R_2} -1\\
					-\frac{t}{R_2} & 0 & - \frac{t}{R_2} +1 & 0
				\end{pmatrix},
				A^3_{p_{-,-}}=\begin{pmatrix}
					0 & \frac{2t-1}{R_1} -1 & 0 & -\frac{t}{R_1}\\
					-\frac{2t-1}{R_1} +1 & 0 & \frac{t}{R_1} & 0\\
					0 & \frac{t}{R_2} & 0 &  -\frac{t}{R_2} -1\\
					-\frac{t}{R_2} & 0 & \frac{t}{R_2} +1 & 0
				\end{pmatrix}.
			\end{equation*}
			\normalsize
			
			Since $A^3_{p_{+,+}}$ is identical to $A^0_{p_{+,+}}$ from Section \ref{subsec:localcam}, $p_{+,+}$ is non-degenerate of elliptic-elliptic type for all values of $t$. Similarly, since $A^3_{p_{-,+}}$ is identical to $A^0_{p_{-,+}}$, $p_{-,+}$ is also non-degenerate of elliptic-elliptic type for all values of $t$.
			
			On the other hand, by direct computation one can see that the characteristic polynomial of $A^3_{p_{+,-}}$ coincides with $P^0_{-,-}(\lambda)$ defined in Section \ref{subsec:localcam}. Then, $p_{+,-}$ is a non-degenerate fixed point of elliptic-elliptic type for all values of $t$. Similarly, the characteristic polynomial of $A^3_{p_{-,-}}$ coincides with $P^0_{+,-}(\lambda)$ of Section \ref{subsec:localcam}, and then $p_{-,-}$ is non-degenerate and of elliptic-elliptic type if $t<t^-$ or $t>t^+$, of focus-focus type if $t^- < t < t^+$ and degenerate if $t\in\{t^-,t^+\}$, where
			
			$$t^\pm = \frac{R_2}{2R_2+R_1\mp 2\sqrt{R_1 R_2}}.$$
		\end{proof}
		
		\begin{figure}[ht!]
			\vspace{20mm}
			\centering
			\includegraphics[scale=0.17]{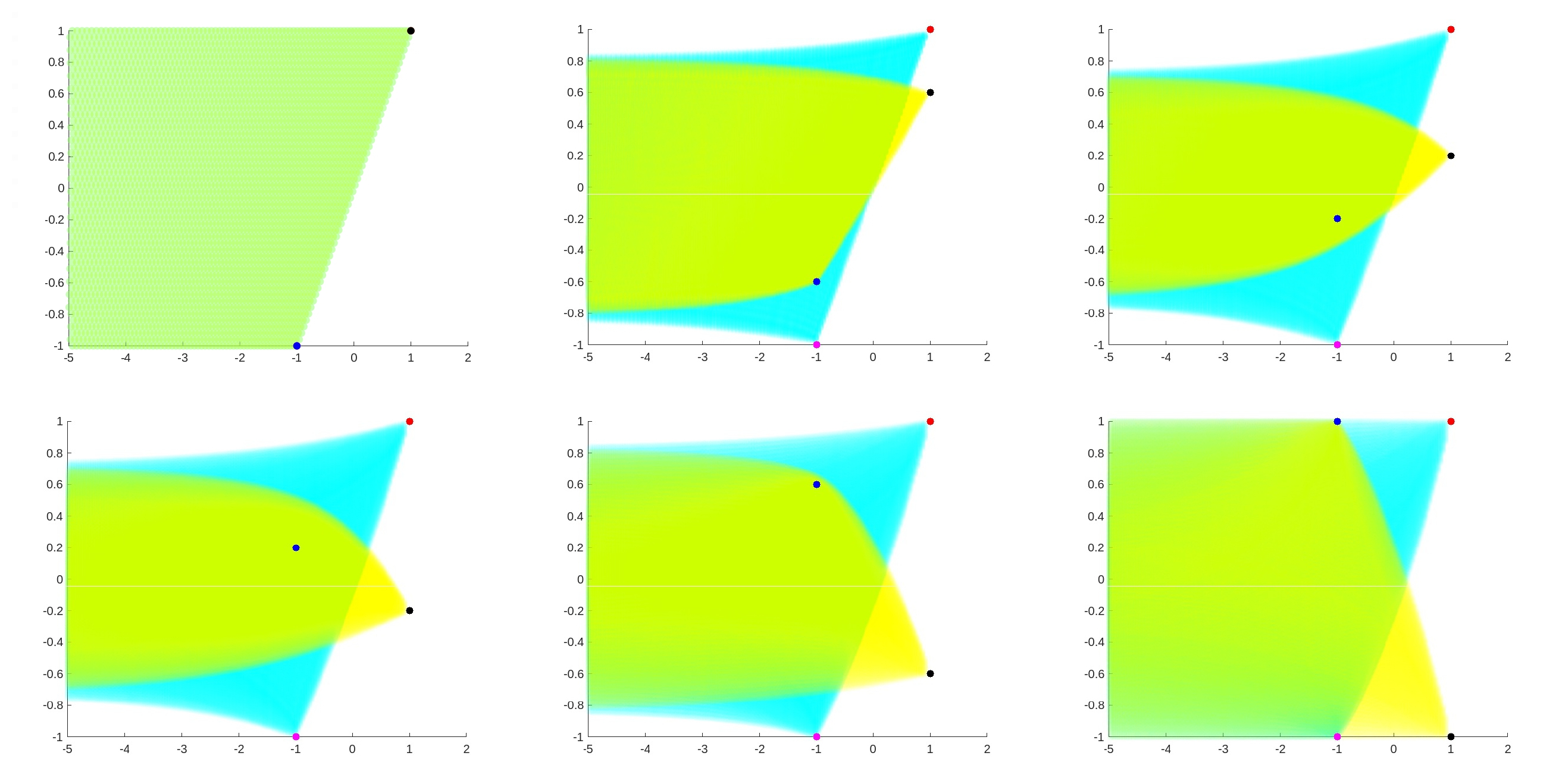}
			\caption{Image of the momentum map of System $3$ for values of $t$ between $0$ (top left) and $1$ (bottom right) by intervals of $0.2$. The images of the four fixed points $p_{+,+},p_{+,-},p_{-,+},p_{-,-}$ are depicted, respectively, in red, black, magenta and blue. The image of the submanifold $\mathbb{S}^2\times \mathbb{S}^2_+$ is depicted in sky blue and the image of the submanifold $\mathbb{S}^2\times \mathbb{S}^2_-$ is depicted in yellow, and both of them have been obtained numerically.}
			\label{fig:cam3}
		\end{figure}
		
		\vfill

		\clearpage
		
		\bibliographystyle{calpha}
		\bibliography{references}
		
	\end{document}